\documentclass{amsart}

\usepackage{amsmath,amssymb,amsthm}

\numberwithin{equation}{section}
\newcommand{\be}{\begin{equation}}
\newcommand{\ee}{\end{equation}}

\newcommand{\pd}{\partial}
\newcommand{\R}{\mathbb R}

\newcommand{\HH}{\mathbb H}

\newcommand{\Hm}{\mathcal H}

\newcommand{\C}{\mathcal C}

\newcommand{\co}{\colon}
\newcommand{\D}{\mathcal D}

\newcommand{\ga}{\gamma}
\newcommand{\la}{\lambda}
\newcommand{\ep}{\varepsilon}
\newcommand{\de}{\delta}

\newcommand{\Sec}{\operatorname{Sec}}
\newcommand{\inj}{\operatorname{inj}}

\newcommand{\dir}{\operatorname{dir}}
\newcommand{\diam}{\operatorname{diam}}
\newcommand{\vol}{\operatorname{vol}}
\newcommand{\area}{\operatorname{area}}

\newcommand{\lcr}{\operatorname{lcr}}

\newtheorem{lemma}{Lemma}[section]
\newtheorem{proposition}[lemma]{Proposition}
\newtheorem{theorem}[lemma]{Theorem}
\newtheorem{corollary}[lemma]{Corollary}

\theoremstyle{definition}
\newtheorem{definition}[lemma]{Definition}
\newtheorem{notation}[lemma]{Notation}

\theoremstyle{remark}
\newtheorem*{remark}{Remark}

\begin{document}

\title{Distance difference representations of Riemannian manifolds}

\author{Sergei Ivanov}
\address{St.~Petersburg Department of Steklov Mathematical Institute of
Russian Academy of Sciences,
Fontanka 27, St.Petersburg 191023, Russia}
\address{Saint Petersburg State University, 
7/9 Universitetskaya emb., St. Petersburg 199034, Russia
}
\email{svivanov@pdmi.ras.ru}

\thanks{Research is supported by the Russian Science Foundation grant 16-11-10039}



\begin{abstract}
Let $M$ be a complete Riemannian manifold and $F\subset M$
a set with a nonempty interior.
For every $x\in M$, let $D_x$ denote the function on $F\times F$
defined by $D_x(y,z)=d(x,y)-d(x,z)$
where $d$ is the geodesic distance in $M$.
The map $x\mapsto D_x$ from $M$ to the space
of continuous functions on $F\times F$,
denoted by $\mathcal D_F$, is called a distance difference representation of~$M$.
This representation, introduced recently by M.~Lassas and T.~Saksala,
is motivated by geophysical imaging among other things.

We prove that the distance difference representation $\mathcal D_F$
is a locally bi-Lipschitz homeomorphism onto
its image $\mathcal D_F(M)$
and that for every open set $U\subset M$
the set $\mathcal D_F(U)$ uniquely determines the Riemannian metric on~$U$.
Furthermore the determination of $M$ from  $\mathcal D_F(M)$ is stable if $M$
has a priori bounds on its diameter, curvature, and injectivity radius.
This extends and strengthens
earlier results by M.~Lassas and T.~Saksala.
\end{abstract}

\maketitle

\section{Introduction}

We study the following geometric inverse problem,
first considered by M.~Lassas and T.~Saksala in \cite{LS}, see also \cite{LS1}.
Let $M=(M^n,g)$ be a complete, connected Riemannian manifold.
%(not necessarily compact).
For $x\in M$, let $D_x\co M\times M\to\R$
be the {\em distance difference function}
defined by
\be\label{e:defD_x}
  D_x(y,z) = d(x,y) - d(x,z), \qquad y,z\in M,
\ee
where $d$ is the geodesic distance in $M$. % determined by~$g$.

Fix a set $F\subset M$, called the \textit{observation domain},
and consider the restrictions $D_x|_{F\times F}$ of distance difference functions to $F\times F$.
%Throughout most of the paper $F$ is assumed open, and in fact one can think 
%that $F$ is a small geodesic ball.
The set of all these restrictions
is called the \textit{distance difference data} for $M$ and~$F$;
%We emphasize that the distance difference data 
we emphasize that it is just a set of functions on $F\times F$
without any additional structure.
The problem is to reconstruct the manifold $M$ and its metric
from this set of functions.
More generally, we consider a variant of the problem
where one knows partial distance difference data
$\{D_x|_{F\times F}\}_{x\in U}$ for some region $U\subset M$
and the goal is to determine the geometry of~$U$.

The notion of distance difference data does not look very natural
but it is a convenient way to represent the following information:
For each $x\in M$, we are given a function on~$F$
that equals the distance function $d(x,\cdot)|_F$ up to an unknown
additive constant.

For a motivation, imagine that $M\simeq D^3$ is the Earth,
the metric $g$ represents the speed of propagation of elastic waves,
and $F=\pd M$ is the Earth's surface.
(Here we generalize our set-up and allow $M$ to have a boundary.)
Assume that micro-earthquakes occur frequently
at unknown points inside the Earth.
When a micro-earthquake occurs at a point $x\in M$
at time~$t$, it produces an elastic wave which
arrives to a point $y\in F$ at time $t+d(x,y)$.
Thus an observer on the surface
can learn the function $t+d(x,\cdot)|_F$.
Since the value of $t$ is unknown, the same information is represented
by the distance difference function $D_x|_{F\times F}$.
Having observed many such events, one collects distance difference
data $\{D_{x_i}|_{F\times F}\}$ for a large set of points $\{x_i\}\subset M$.
%whose locations are unknown.
Idealizing the situation, we assume that one knows the data for all $x\in M$
or at least for all points from 
some ``active'' region $U\subset M$.
The goal is to learn as much as possible about the metric $g$
from these data.
Clearly the best one can hope for is to reconstruct $(M,g)$ up to
a Riemannian isometry.

Like in \cite{LS} and \cite{LS1}, we actually consider a simpler variant
of the problem where the manifold $M$ has no boundary and 
the observation domain $F$ has a nonempty interior.
In some cases the problem for a manifold with boundary can be reduced
to the simpler variant by means of a
suitable extension of the manifold, see Section~\ref{sec:boundary}.

We refer to \cite{LS,LS1} %\cite{KKL,Kyr,LS,LS1}
for more thorough discussion of distance
difference representations and their connections to a variety of inverse problems.
See also \cite{KKL,Kyr} for discussion and applications of
the more basic \textit{distance representation},
that is a map $\mathcal R_F\co M\to \C(F)$
defined by $\mathcal R_F(x)=d(x,\cdot)|_F$ for all $x\in M$.
%Since $\D_F(x)$ is trivially determined by $\mathcal R_F(x)$,
%the 
All results of this paper concerning $\D_F$ trivially imply similar statements about $\mathcal R_F$.

\subsection*{Definitions and statement of the results}
Now we proceed with formal definitions and statements.
For a complete, connected, boundaryless Riemannian manifold $M=(M^n,g)$ and a set $F\subset M$
define a map $\D_F\co M\to \C(F\times F)$ by 
$$
\D_F(x)=D_x|_{F\times F}
$$
where $D_x$ is the distance difference function defined by \eqref{e:defD_x}.
The target space $\C(F\times F)$ of $\D_F$ is
the space of continuous functions on $F\times F$
equipped with the sup-norm distance:
$$
\|f_1-f_2\|=\sup_{y,z\in F}\{|f_1(y,z)-f_2(y,z)|\}, \qquad f_1,f_2\in\C(F\times F) .
$$
The map $\D_F$ is called the \textit{distance difference representation} of $M$.

The triangle inequality implies that
$$
 \|\D_F(x)-\D_F(y)\| \le 2d(x,y)
$$
for all $x,y\in M$. 
Thus $\D_F$ is a 2-Lipschitz map.
Note that the norm $\|\D_F(x)\|$ may be infinite if $F$ is unbounded
but the distance $\|\D_F(x)-\D_F(y)\|$ is always finite.

Our first goal is to show that the set $\D_F(M)\subset\C(F\times F)$ is a homeomorphic image of~$M$.
In particular this implies that the distance difference data determines the topology of~$M$.
This was previously known in the case when $M$ is compact, see \cite[Theorem 2.3]{LS}.
In fact, the proof in \cite{LS} shows that $\D_F$ is injective, and by compactness of~$M$ it follows
that $\D_F$ is a homeomorphism onto its image.
In Theorem \ref{t:inverse} below we extend this result to the non-compact case.
Moreover we show that the inverse map $\D_F^{-1}$ is locally Lipschitz.
%From the ``practical'' point of view 
%the local Lipschitz continuity of $\D_F^{-1}$ means than an observer can reliably
%distinguish between distance difference data
%coming from an ``interesting'' region $U\subset M$ and
%those coming from other points (that may be located far away from~$U$).

\begin{theorem}\label{t:inverse}
Let $M^n$ be a complete, connected Riemannian manifold without boundary,
$n\ge 2$, and let $F\subset M$ be a subset with a nonempty interior.
Then the map $\D_F\co M\to\C(F\times F)$ is a homeomorphism onto its image
and the inverse map $\D_F^{-1}\co\D_F(M)\to M$ is locally Lipschitz.
\end{theorem}

The local Lipschitz constant of $\D_F^{-1}$ at a point $\D_F(x_0)$
(and the radius of a neighborhood where this constant works)
can be estimated in terms of geometric parameters
such as the size of~$F$, the distance from $x_0$ to~$F$,
curvature and injectivity radius bounds in a suitably large ball.
For details, see Propositions \ref{p:holder} and \ref{p:bilip} in Section~\ref{sec:inverse}.
In the compact case the effective version of Theorem \ref{t:inverse}
is stated as Corollary~\ref{cor:bilip} below.
%First let us introduce some notation.

We denote by $\diam(M)$ the diameter, $\Sec_M$ the sectional curvature,
and $\inj_M$ the injectivity radius of a Riemannian manifold~$M$.
For $n\ge 2$ and $D,K,i_0>0$ we denote by $\mathcal M(n,D,K,i_0)$
the class of all compact boundaryless Riemannian $n$-manifolds $M$ 
with $\diam(M)\le D$, $|\Sec_M|\le K$ and $\inj_M\ge i_0$.

\begin{corollary}\label{cor:bilip}
Let $M\in\mathcal M(n,D,K,i_0)$ and assume that a set $F\subset M$ contains a ball of radius~$\rho_0$.
Then the map $\D_F\co M\to \D_F(M)\subset\C(F\times F)$ is a bi-Lipschitz homeomorphism
with a bi-Lipschitz constant determined by $n$, $D$, $K$, $i_0$, $\rho_0$.
\end{corollary}

Now consider the problem of unique determination of a Riemannian metric from
the distance difference data.
Our informal set-up is the following.
There is an unknown Riemannian manifold $M$ and an unknown open set $U\subset M$.
We know the topology and differential structure of the observation domain $F$
and we are given the subset $\D_F(U)$
of $\C(F\times F)$. 
Our result is that these data determine the geometry of the domain $(U,g|_U)$ uniquely
up to a Riemannian isometry.
The precise statement is the following.

\begin{theorem}\label{t:reconstruction}
Let $M_1=(M_1^n,g_1)$ and $M_2=(M_2^n,g_2)$ be complete, connected, boundaryless Riemannian manifolds
and $n\ge 2$.
Assume that $M_1$ and $M_2$  share a nonempty subset $F$ which 
is open in both manifolds, and 
they induce the same topology and the same differential structure on $F$.
% and $g_1|_F=g_2|_F$.
For $i=1,2$, let $\D_F^i$ denote the distance difference representation
of $M_i$ in $\C(F\times F)$.

Let $U_1\subset M_1$ and $U_2\subset M_2$ be open sets such that
the subsets $\D_F^1(U_1)$ and $\D_F^2(U_2)$
of $\C(F\times F)$ coincide.
Then the map $\phi \co U_1\to U_2$ defined by
\be\label{e:def-phi}
 \phi=(\D_F^2)^{-1}\circ \D_F^1|_{U_1}
\ee
is a Riemannian isometry between $(U_1,g_1|_{U_1})$ and $(U_2,g_2|_{U_2})$,
that is, $\phi$ is a diffeomorphism between $U_1$ and $U_2$ and $\phi_*g_1=g_2$.
\end{theorem}

Theorem \ref{t:reconstruction} applied to $U=M$ yields that
the distance difference data $\D_F(M)$ determine the isometry type of $(M,g)$.
In the case when $M$ is compact this was known due to M.~Lassas and T.~Saksala \cite{LS}.
In \cite{LS1} the result was extended to the following situation:
$U$ is a pre-compact region bounded by a smooth hypersurface $\pd U$
and $F$ is a neighborhood of $\pd U$ in $M\setminus U$.
The proofs in \cite{LS} and \cite{LS1} are based on the same approach:
first it is shown that the two Riemannian manifolds in question are projectively equivalent
and then certain delicate properties of projective equivalences
imply that they are actually isometric.
The last part requires the knowledge
of global distance difference data and it does not work in the localized setting
of Theorem~\ref{t:reconstruction} (where $U$ can be separated away from~$F$).
Another question left open in \cite{LS} and \cite{LS1}
is whether the determination of metric is stable with respect to variations
of the distance difference data.

Our proof of Theorem \ref{t:reconstruction} is based on a different
technique, which works as well for metric spaces with locally bounded curvature
in the sense of Alexandrov.
With a standard compactness argument, this implies
stability of metric determination within the class $\mathcal M(n,D,K,i_0)$:
If two manifolds $M_1$ and $M_2$ from this class
share an open set $F$ and their distance difference data
are Hausdorff-close subsets of $\C(F\times F)$, then the manifolds
are close in Gromov--Hausdorff (and hence in $C^{1,\alpha}$) topology.
See Proposition \ref{p:stability} for details.

\subsection*{Organization of the paper}
The proofs of the theorems are based mainly on distance comparison inequalities
implied by Toponogov's theorem.
These inequalities are collected in Section \ref{sec:toponogov}.
Another key ingredient of the proof is a minimizing geodesic extension
property, see Proposition \ref{p:extension}.
This proposition, which may be of independent interest,
provides a lower bound on the length of a minimizing extension
of a geodesic beyond a non-cut point in terms of the length of
a minimizing extension beyond the other endpoint.
The proof of  Proposition \ref{p:extension} is contained in sections
\ref{sec:extension} and~\ref{sec:sasaki}.
In Section \ref{sec:inverse} we prove Theorem \ref{t:inverse}
and Corollary \ref{cor:bilip}.
In Section \ref{sec:reconstruction} we prove Theorem \ref{t:reconstruction}
and observe that the proof implies stability of the metric determination,
see Proposition \ref{p:stability}.
In Section \ref{sec:boundary} we obtain some partial generalizations
of our results to manifolds with boundaries.

\subsection*{Acknowledgements}
The author thanks Matti Lassas, Yaroslav Kurylev, Charles Fefferman, and Hariharan Narayanan
for discussions that attracted the author's attention to the problem and inspired some ideas of the proofs.

\section{Distance comparison estimates}
\label{sec:toponogov}

In this section we set up notation and prepare some tools
from comparison geometry.
These are pretty standard implications of Toponogov's comparison theorem
and they hold in all Alexandrov spaces of curvature bounded below,
see \cite{BBI,BGP}.
We include all formulations and proofs since we need 
to handle local curvature bounds.

Let $M=(M^n,g)$, $n\ge 2$, be a complete, connected, boundaryless Riemannian manifold.
We use notation $|xy|$ for the distance $d(x,y)$
between points $x,y\in M$.
We denote by $B_r(x)$ the geodesic
ball of radius $r>0$ centered at~$x\in M$,
by $L(\ga)$ the length of a path $\ga$ in~$M$,
and by $SM$ the unit tangent bundle of~$M$.

For $x,y\in M$, we denote by $[xy]$ any minimizing geodesic connecting $x$ and~$y$.
In cases when a minimizing geodesic is not unique, we assume that some choice of $[xy]$ is fixed
for each pair $x,y$.
For $x,y,z\in M$ we denote by $\angle yxz$ the angle between the directions of
$[xy]$ and $[xz]$ in $T_xM$.

In order to handle local curvature bounds 
we introduce the following quantity.

\begin{definition}\label{d:lcr}
For $x\in M$, we define the \textit{lower curvature radius} at $x$,
denoted by $\lcr(x)$ or $\lcr_M(x)$, as the supremum of all $r>0$
such that $\Sec_M\ge -1/r^2$ everywhere in the ball $B_{2r}(x)$.
\end{definition}

Clearly $x\mapsto\lcr(x)$
is a positive 1-Lipschitz function on~$M$.
The lower curvature radius naturally rescales with the metric.
That is, for a rescaled Riemannian manifold $\la M:=(M,\la^2g)$, where $\la>0$,
one has $\lcr_{\la M}(x)=\la\cdot\lcr_M(x)$.

We need the following variant of Toponogov's comparison theorem.

\begin{proposition}\label{p:toponogov}
Let $x,y,z\in M$ and $K>0$ be such that
$\Sec_M\ge-K$ everywhere in the ball $B_{2r}(x)$
where $r=\max\{|xy|,|xz|\}$.
Let $\bar x,\bar y,\bar z$ be points in the rescaled hyperbolic plane $K^{-1/2}\HH^2$
such that $|\bar x\bar y|=|xy|$, $|\bar x\bar z|=|xz|$,
and $\angle\bar y\bar x\bar z=\angle yxz$.
Then $|yz|\le |\bar y\bar z|$.

In particular, if $\lcr(x)\ge 1\ge \max\{|xy|,|xz|\}$ then the inequality $|yz|\le |\bar y\bar z|$
holds for $\bar x,\bar y,\bar z$ constructed as above in the standard hyperbolic plane $\HH^2$.
\end{proposition}

\begin{proof}
The first part is essentially the hinge version of Toponogovs's theorem, see \cite[Ch.~11, Theorem 70]{Pe}.
The difference is that in the standard formulation of Toponogov's theorem
the curvature bound $\Sec_M\ge-K$
is assumed everywhere on~$M$. 
However from the proof in e.g.\ \cite{Pe} one can see that
the curvature bound needs to hold only on the union of all minimizing
geodesics from $y$ to points of~$[xz]$.
Thus the refined formulation works as well.

The second part follows from the first one and the definition of $\lcr(x)$
by setting $K=1$.
\end{proof}

In the rest of this section we deduce several inequalities used throughout the paper.

\begin{lemma}\label{l:1var-ineq}
For any $p,x,y\in M$,
$$
  |py| \le |px| - |xy|\cos\angle pxy + C \frac{|xy|^2}{ \min\{ |px|, \lcr(x) \} }
$$
where $C>0$ is an absolute constant.
\end{lemma}

\begin{proof}
Let $\alpha=\angle pxy$.
By rescaling we may assume that $\min\{ |px|, \lcr(x) \}=1$.
Thus $|px|\ge 1$, $\lcr(x)\ge 1$, and we need to prove that
\be\label{e:1var-ineq1}
  |py| \le |px| - |xy|\cos\alpha + C |xy|^2 .
\ee
We may also assume that $|xy|<1/2$, otherwise \eqref{e:1var-ineq1}
holds for any $C\ge 4$ due to the triangle inequality $|py|\le|px|+|xy|$.

Let $p_1\in[px]$ be such that $|p_1x|=1$.
Pick $\bar p_1,\bar x,\bar y\in\HH^2$
such that $|\bar p_1\bar x|=1$, $|\bar x\bar y|=|xy|$ and $\angle\bar p_1\bar x\bar y=\alpha$.
By Proposition \ref{p:toponogov} we have $|p_1y| \le |\bar p_1\bar y|$.

Consider the distance function $f=|\bar p_1 \cdot|$ on $\HH^2$.
This function is smooth on $\HH^2\setminus\{\bar p_1\}$
and its derivative at $\bar x$ in the direction of $[\bar x\bar y]$ equals $-\cos\alpha$.
Thus
$$
 |p_1y| \le |\bar p_1\bar y| 
 \le |\bar p_1\bar x| - |\bar x\bar y|\cos\alpha + C |\bar x\bar y|^2
 = |p_1x| - |xy|\cos\alpha + C |xy|^2
$$
where $C$ is an absolute constant determined by the maximum second derivative
of~$f$ on $B_2(\bar p_1)\setminus B_{1/2}(\bar p_1)$.
This and the triangle inequality $|py|\le |pp_1|+|p_1y|$ imply \eqref{e:1var-ineq1}
and the lemma follows.
\end{proof}

\begin{lemma}\label{l:short-median}
Let $p,q,x,y\in M$ be such that $x\in [pq]$ and 
$$
r:=\min\{|px|,|qx|,\lcr(x)\} > 0 .
$$
Then
\be\label{e:1stvar}
   \bigl| |px| - |py| - |xy|\cos\angle pxy \bigr| \le  C \frac{|xy|^2}r .
\ee
and
\be\label{e:short-median}
  |py|+|yq| \le |pq| + 2C \frac{|pz|^2}{r}
\ee
where $C>0$ is the absolute constant from
Lemma \ref{l:1var-ineq}.
\end{lemma}

\begin{proof}
Let $\alpha=\angle pxy$, then $\angle qxy=\pi-\alpha$.
By Lemma \ref{l:1var-ineq} we have
\be\label{e:median1}
 |py| \le |px| - |xy|\cos\alpha + C \frac{|xy|^2}{r}
\ee
and, using $q$ in place of $p$,
\be\label{e:median2}
 |qy| \le |qx| + |xy|\cos\alpha + C \frac{|xy|^2}{r} .
\ee
Summing \eqref{e:median1} and \eqref{e:median2} yields \eqref{e:short-median}.
Subtracting \eqref{e:median2} from the triangle inequality $|py|+|qy|\ge |pq|$
yields
$$
  |py| \ge |px| - |xy|\cos\alpha - C \frac{|xy|^2}{r} .
$$
This and \eqref{e:median1} implies \eqref{e:1stvar}.
\end{proof}

\begin{lemma}\label{l:shortcut}
Let $x,y,z\in M$ and $\alpha=\pi-\angle yxz$.
Then
$$
  |xy|+|xz|-|yz| \ge c \alpha^2 \min\{|xy|,|xz|,\lcr(x)\}
$$
where $c>0$ is an absolute constant.
\end{lemma}

\begin{proof}
By rescaling we may assume that $\min\{|xy|,|xz|,\lcr(x)\}=1$.
Let $y_1\in[xy]$ and $z_1\in[xz]$ be the points at distance $1$ from~$x$.
The triangle inequality $|yz|\le|yy_1|+|y_1z_1|+|z_1z|$ implies that
$$
  |xy|+|xz|-|yz| \ge |xy_1|+|xz_1|-|y_1z_1| =2-|y_1z_1|.
$$
Thus it suffices to show that $|y_1z_1|\le 2-c\alpha^2$
for a suitable constant $c>0$.

Pick $\bar x,\bar y_1,\bar z_1\in\HH^2$ such that
$|\bar x\bar y_1|=|\bar x\bar z_1|=1$ and $\angle \bar y_1\bar x\bar z_1=\pi-\alpha$.
By Proposition \ref{p:toponogov} we have $|y_1z_1|\le |\bar y_1\bar z_1|$.
It remains to show that $|\bar y_1\bar z_1|\le 2-c\alpha^2$.
This follows from the hyperbolic law of cosines
$$
 \cosh|\bar y_1\bar z_1| = \cosh^2 1 - \sinh^2 1 \cdot \cos\alpha
 = \cosh 2 - \sinh^2 1 \cdot (1-\cos\alpha)
$$
and an elementary inequality $1-\cos\alpha \ge \frac2{\pi^2} \alpha^2$.
\end{proof}

\begin{lemma}\label{l:lipexp}
Let $x,y,z\in M$.

1. If, for some $K,R>0$, one has $|xy|\le R$ and $\Sec_M\ge-K$ everywhere in the ball $B_{2R}(x)$,
then 
\be\label{e:lipexp}
  |yz| \le C_{K,R}\cdot \angle yxz + \bigl| |xy|-|xz| \bigr| .
\ee
where $C_{K,R}>0$ is determined by $K$ and $R$.

2. If $|xy|\le\lcr(x)$, then
\be\label{e:lipexp1}
  |yz| \le 2|xy|\cdot \angle yxz + \bigl| |xy|-|xz| \bigr| .
\ee
\end{lemma}

\begin{proof}
1. Let $\alpha=\angle yxz$.
We may assume that $|xy|\le|xz|$, otherwise swap $x$ and $y$.
Let $z_1\in[xz]$ be such that $|xz_1|=|xy|$.
Pick points $\bar x,\bar y,\bar z_1$
in the rescaled hyperbolic plane $K^{-1/2}\HH^2$
such that $|\bar x\bar y|=|\bar x\bar z_1|=|xy|$ and $\angle \bar y\bar x\bar z_1=\alpha$.
By Proposition \ref{p:toponogov} we have 
$$
|yz_1|\le |\bar y\bar z_1| \le C_{K,R}\cdot\alpha
$$
where $C_{K,R}$ equals $1/2\pi$ times the length of the circle of radius $R$ in $K^{-1/2}\HH^2$.
Therefore
$$
 |yz| \le |yz_1|+|z_1z|  \le C_{K,R}\cdot\alpha + |xz|-|xy|
$$
by the triangle inequality and the construction of~$z_1$.

2. Rescale the metric by the factor $|xy|^{-1}$ and observe that
in the rescaled space the assumptions of the first part
are satisfied for $K=R=1$. Hence \eqref{e:lipexp} holds
with the constant $C_{1,1}=\sinh1\le 2$.
Rescaling the metric back yields \eqref{e:lipexp1}.
\end{proof}

\begin{notation}\label{n:dir}
For $x\in M$ and a set $B\subset M$,
we denote by $\dir(x,B)$ or $\dir_M(x,B)$ the subset
of the unit sphere $S_xM\subset T_xM$
consisting of the initial directions of all minimizing geodesics
from $x$ to points of~$B$.
\end{notation}

\begin{lemma}\label{l:dirvol}
For any $K,R>0$ there exist $\la_{K,R}>0$ such that
the following holds.
Let $x\in M$ and let $B\subset B_R(x)$ be a measurable set.
Assume that $\Sec_M\ge-K$ everywhere in the ball $B_{2R}(x)$. Then
\be\label{e:dirvol}
 \Hm^{n-1}(\dir(x,B)) \ge \la_{K,R}^{n-1} \cdot \frac{\vol(B)}{\diam(B)}
\ee
where $\Hm^{n-1}$ denotes the $(n-1)$-dimensional Hausdorff measure.
\end{lemma}

\begin{proof}
Let $K,R,x,B$ be as above.
Let $d_{\min}$ and $d_{\max}$ denote the supremum and infimum of distances
from $x$ to points of~$B$. Then $d_{\max}-d_{\min}\le\diam(B)$.
For $r>0$, denote by $S(r)$ the sphere of radius $r$ centered at~$x$,
i.e., $S(r)=\{y\in M: |xy|=r\}$.
By the co-area formula for Lipschitz functions (see e.g.\ \cite[Theorem 3.2.11]{Federer})
applied to the distance function of $x$,
$$
 \vol(B) =  \int_{d_{\min}}^{d_{\max}} \Hm^{n-1}(B\cap S(r))\, dr .
$$
Hence there exists $r\in [d_{\min},d_{\max}]$ such that
\be\label{e:slicevol}
  \Hm^{n-1}(B\cap S(r))\ge \frac{\vol(B)}{d_{\max}-d_{\min}}  \ge \frac{\vol(B)}{\diam(B)} .
\ee
Now consider the set $\Sigma=\dir(x,B\cap S(r))\subset\dir(x,B)$
and the map $v\mapsto\exp_x(rv)$ from $\Sigma$ onto $B\cap S(r)$.
By Lemma \ref{l:lipexp}, this map is $C_{K,R}$-Lipschitz, hence
$$
 \Hm^{n-1}(B\cap S(r)) \le C_{K,R}^{n-1} \Hm^{n-1}(\Sigma) .
$$
This and \eqref{e:slicevol} imply \eqref{e:dirvol}
for $\la_{K,R}=C_{K,R}^{-1}$.
\end{proof}

\section{Minimizing geodesic extensions}
\label{sec:extension}

The main result of this section in Proposition \ref{p:extension},
which is a quantitative version of the fact that the cut point
relation in a Riemannian manifold is symmetric.
Recall that a point $x\in M$ is a \textit{cut point} of $p\in M$
if a minimizing geodesic $[px]$ cannot be extended as a minimizing geodesic
beyond~$x$. This can occur in two cases:
either there are two distinct minimizing geodesics from $p$ to $x$,
or $p$ and $x$ are conjugate points along $[px]$. Both properties are symmetric
with respect to $p$ and $x$, hence so is the cut point relation.
Thus, if $[px]$ admits a minimizing extension beyond $p$, then it
also admits a minimizing extension beyond $x$.
Proposition \ref{p:extension} provides a lower bound on the length of this extension.
%It turns out that this lower bound is linear on the length of the extension beyond $p$,
%with the coefficient determined by geometric bounds of the Riemannian metric.

\begin{proposition}\label{p:extension}
For any $K,R_0>0$ there exists $\la_0=\la_0(K,R_0)>0$ such that the following holds.
Let $M$ be a complete Riemannian manifold,
$p,x\in M$, $r_0>0$,
and assume the following geometric bounds:
\begin{enumerate}
 \item $|px|\le R_0$;
 \item $|\Sec_M|\le K$ everywhere in the ball $B_{R_0+3r_0}(p)$;
 \item $\inj_M(x)\ge r_0$.
\end{enumerate}
Suppose that $p$ belongs to a minimizing geodesic $[qx]$ 
for some $q\in M$.
Then there exists a minimizing extension $[py]$ of $[px]$
such that 
\be\label{e:extension}
  |xy| \ge \la_0 \min\{ |pq| , r_0, K^{-1/2} \} .
\ee
\end{proposition}

Our main application of Proposition \ref{p:extension} is the following.

\begin{corollary}\label{cor:1stvar}
For any $K,R_0,r_0>0$ there exists $\Lambda=\Lambda(K,R_0,r_0)>0$ such that the following holds.
Let $M$ be a complete Riemannian manifold, $p,x\in M$,
and assume the bounds (1)--(3) from Proposition \ref{p:extension}.

Suppose that $p$ belongs to a minimizing geodesic $[qx]$ 
for some $q\in M$.
Then, for every $z\in M$,
\be\label{e:cor1stvar}
  \bigl| |px| - |pz| - |xz|\cos\angle pxz \bigr| 
  \le  \frac{\Lambda |xz|^2}{\min\{ |px|,|pq|, 1 \}} .
\ee
\end{corollary}

First we deduce Corollary \ref{cor:1stvar} from Proposition \ref{p:extension}.

\begin{proof}[Proof of Corollary \ref{cor:1stvar}]
Let $[py]$ be a minimizing extension of $[px]$ provided by Proposition \ref{p:extension}.
We apply the first part of Lemma \ref{l:short-median} to $p,y,x,z$ and obtain
\be\label{e:cor1stvar1}
  \bigl| |px| - |pz| - |xz|\cos\angle pxz \bigr| 
  \le  \frac{C |xz|^2}{\min\{ |xy|,|px|,\lcr(x) \}}
\ee
where $C$ is an absolute constant.
The assumption (2) of Proposition \ref{p:extension} implies that
$\lcr(x)\ge\min\{K^{-1/2},r_0\}$.
This inequality, \eqref{e:extension} and
\eqref{e:cor1stvar1} imply \eqref{e:cor1stvar}
for a suitable $\Lambda$ determined by $K$, $r_0$, and $\la_0$.
\end{proof}

The proof of Proposition \ref{p:extension} occupies the rest of this section
and the next section.
In this section we prove Proposition \ref{p:extension} assuming a technical
estimate stated as Lemma \ref{l:beta-vs-alpha}.
The next section is devoted to the proof of Lemma \ref{l:beta-vs-alpha}.

First observe that the statement of Proposition \ref{p:extension} is scale invariant.
We replace $M$ by a rescaled manifold $\la M$ where $\la=\max\{r_0^{-1},K^{1/2}\}$
and change the parameters $K,R_0,r_0$ accordingly.
The resulting manifold, for which we reuse the notation $M$,
satisfies $|\Sec_M|\le 1$ in the ball $B_{R_0+3}(p)$ and $\inj_M(x)\ge 1$.
The desired inequality \eqref{e:extension} now takes the form
\be\label{e:extension1}
 |xy| \ge \la_0 \min\{ |pq| , 1 \} .
\ee
Thus it suffices to prove the proposition for $K=r_0=1$.
This is assumed throughout the rest of the proof.
Our goal is to prove that there is a minimizing extension
$[py]$ of $[px]$ satisfying \eqref{e:extension1}
for a suitable $\la_0$ determined by $R_0$.
Since $|\Sec_M|\le 1$ in $B_{R_0+3}(p)$,
the lower curvature radius (see Definition \ref{d:lcr})
is bounded below by~1 in the ball $B_{R_0+1}(p)$.
All points appearing in the proof belong to this ball.

Let $p,q,x$ be as in the formulation of 
Proposition \ref{p:extension}.
We may assume that
$|px| > 1/2$.
Indeed, since $\inj_M(x)\ge 1$ and $|\Sec_M|\le 1$ on $B_1(x)$,
any geodesic of length~1 with midpoint at~$x$ is minimizing.
%(This can be seen e.g.\ by comparing the metric in $B_1(x)$ 
%with that of the standard $n$-sphere, see \cite[Ch.~6, Theorem~27]{Pe}.)
Thus if $|px|\le 1/2$ then
there is a minimizing extension $[py]$ of $[px]$ with $|xy|=1/2$
and \eqref{e:extension1} holds for any $\la_0\le 1/2$.

Define $L=|px|$ and
$$
 r_1 = \tfrac1{10}\min\{|pq|, 1 \} .
$$
Since $|px| > 1/2$, we have
\be\label{e:L-vs-r1}
  L > 1/2 \ge 5r_1 .
\ee
Let $\ga\co[0,+\infty)\to M$ be the unit-speed geodesic
starting at $p$ and containing $[px]$,
in particular $\ga(0)=p$ and $\ga(L)=x$.
Pick $r\in(0,r_1]$, let $y=\ga(L+r)$,
and assume that the geodesic arc $\ga|_{[0,L+r]}=[px]\cup[xy]$
is not minimizing.
We are going to show that
this non-minimality implies a lower bound on $r/r_1$,
see \eqref{e:case1} below.

Let $L_1=|py|$. Fix a minimizing geodesic $[py]$ and parametrize 
it by arc length as $\ga_1\co[0,L_1]\to M$.
Let $\alpha$ and $\beta$ be the angles between $\ga$ and $\ga_1$ at $p$ and~$y$,
that is,
$\alpha=\angle(\dot\ga(0),\dot\ga_1(0))$ and $\beta=\angle(\dot\ga(L+r),\dot\ga_1(L_1))$.
Since  $[px]\cup[xy]$ is a non-minimizing geodesic,
we have $L_1<L+r$ and $\alpha,\beta>0$.
Also note that 
\be\label{e:L1-vs-r1}
L_1\ge L-r \ge 4r_1 \ge 4r
\ee
by the triangle inequality, \eqref{e:L-vs-r1},
and the inequality $r\le r_1$.

\begin{lemma}\label{l:beta-vs-alpha}
Under the above assumptions,
$%\be\label{e:beta-vs-alpha}
\beta\le C_1\alpha
$ %\ee
%for some $C_1=C_1(R_0)$.
where $C_1$ is a positive number determined by~$R_0$.
\end{lemma}

The proof Lemma \ref{l:beta-vs-alpha}
is given in Section~\ref{sec:sasaki}.
%To convince yourself that this lemma is correct,
%observe that $\beta$ is not very different from the angle
%between tangent vectors $\dot\ga(L_1)$ and $\dot\ga_1(L_1)$
%measured in normal geodesic coordinates centered at~$x$.
%These vectors are images of $\dot\ga(0)$ and $\dot\ga_1(0)$
%under the time $L_1$ map of the geodesic flow
%and this map has a Lipschitz constant controlled by the length $L_1$ and
%the curvature bounds.
%The value of $C_1$ in Lemma \ref{l:beta-vs-alpha} is essentially
%this Lipschitz constant.
%See Section~\ref{sec:sasaki} for formal details.
%
Here we prove Proposition \ref{p:extension} assuming
that Lemma \ref{l:beta-vs-alpha} holds true.
By \eqref{e:L1-vs-r1} we have $L(\ga_1)=L_1>2r$,
hence there is a point $z=\ga_1(L_1-2r)$ on $\ga_1$
such that $|xz|=2r$ and $|pz|=L_1-2r$.
Since $[pq]\cup[px]$ is a minimizing geodesic, the triangle
inequality implies that
$$
 |pq|+L = |qx| \le |qz| + |xz| .
$$
We add $|pz|$ to both sides of this inequality and rewrite it as
\be\label{e:pqzx}
 |pq|+|pz|-|qz| \le |xz|+|pz|-L = |xz|+L_1-2r-L < |xz|-r .
\ee
Here the last inequality follows from the fact that $L_1<L+r$.

%Furthermore,
%\be\label{e:pz-bounds}
%  2r_1\le |pz| = L_1-2r  < L-r 
%\ee
%since $L_1<L+r$.
%Our plan is to compare the lengths of $[px]=[pq]\cup[qx]$
%and the broken geodesic $[pz]\cup[zx]$
%whose first leg $|pz|$ is bounded by $|pq|+|qz|$ minus
%a shortcut that we can make at~$p$.
%The second leg $|zx|$ is estimated differently
%for small and large values of $\alpha$.

Lemma \ref{l:shortcut} applied to $p,q,z$ yields the
following estimate on the left-hand side of \eqref{e:pqzx}:
\be\label{e:pq+pz-qz}
|pq|+|pz|-|qz| \ge c\alpha^2 \min\{|pq|,|pz|,\lcr(p)\} .
\ee
Observe that $|pq|\ge r_1$ by the definition of $r_1$,
$|pz|=L_1-2r\ge r_1$ by \eqref{e:L1-vs-r1},
and $\lcr(p)\ge 1\ge r_1$.
Thus the minimum in \eqref{e:pq+pz-qz} is bounded below by $r_1$
and we obtain %that
$$
|pq|+|pz|-|qz| \ge c\alpha^2r_1 .
$$
This inequality and \eqref{e:pqzx} imply the estimate
\be\label{e:xz-lower}
 |xz| - r > c\alpha^2r_1 .
\ee
%where $c>0$ is the absolute constant from Lemma \ref{l:shortcut}.

Now we estimate $|xz|-r$ from above.
Let $x_1$ be the midpoint of $[zy]$.
Note that $|x_1z|=|x_1y|=r=|xy|$.
The second part of Lemma \ref{l:lipexp} applied to $y,x,x_1$
implies that
\be\label{e:xx1-bound}
|xx_1|\le 2\beta r \le 2C_1\alpha r .
\ee
where the second inequality follows from Lemma \ref{l:beta-vs-alpha}.
Applying the second part of Lemma \ref{l:short-median}
to points $y,z,x_1,x$ 
and taking into account \eqref{e:xx1-bound}
yields that
$$
 |xz|-r = |xz|+|xy|-|yz| \le \frac{2C |xx_1|^2}{\min\{ |x_1y|, |x_1z|, \lcr(x_1) \}} 
 = 2C \frac{|xx_1|^2}{r}\le C_2\alpha^2 r
$$
where $C_2=8CC_1^2$.
This and \eqref{e:xz-lower} imply that 
%$C_2\alpha^2 r > c\alpha^2r_1$
%and we conclude that
\be\label{e:case1}
  r > cC_2^{-1} r_1 .
\ee
Recall that this holds for any $r\in(0,r_1]$ 
such that  $\ga|_{[0,L+r]}$ is not minimizing.

%Now we are in a position to prove Proposition \ref{e:extension}.
Now define $\la_0= \frac1{10}\min \{1, cC_2^{-1} \}$
and apply the above argument to %the value of $r$ given by
$$
  r  = 10\la_0 r_1 = \la_0\min\{|pq|,1\}.
$$
If the geodesic arc $\ga|_{[0,L+r]}=[px]\cup[xy]$
is not minimizing then \eqref{e:case1} holds and hence $10\la_0>cC_2^{-1}$,
contrary to the definition of~$\la_0$.
This contradiction 
shows that $[px]\cup[xy]$ is a minimizing extension of $[px]$
with  $|xy|=\la_0\min\{|pq|,1\}$, so \eqref{e:extension1} is satisfied.
This finishes the proof of Proposition \ref{e:extension}.

\section{Proof of Lemma \ref{l:beta-vs-alpha}}
\label{sec:sasaki}

Here we prove Lemma \ref{l:beta-vs-alpha}
left from the previous section.
Recall that we are dealing with
two unit-speed geodesics $\ga\co[0,L+r]\to M$ and $\ga_1\co[0,L_1]\to M$
between points $p=\ga(0)=\ga_1(0)$
and $y=\ga(L+r)=\ga_1(L_1)$.
We have
\begin{enumerate}
 \item $|L-L_1|\le r\le 1/10$,
 \item $|\Sec_M|\le 1$ in the ball $B_{R_0+3}(p)$
where $R_0$ is an upper bound for~$L$.
 \item $\inj_M(x)\ge 1$ where $x=\ga(L)$.
\end{enumerate}
%Other details from Section~\ref{sec:extension}
%are not needed here.
Our goal is to show that
\be\label{e:beta-goal}
 \beta\le C_1\alpha
\ee
where $\alpha=\angle(\dot\ga(0),\dot\ga_1(0))$,
$\beta=\angle(\dot\ga(L+r),\dot\ga_1(L_1))$
and $C_1$ is a positive number determined by~$R_0$.

Throughout this section we denote by $C$ various absolute positive constants.
The precise value of $C$ may change between the formulas
(and even within one formula).
By $C_1$, $C_2$, etc, we denote positive constants depending on $R_0$
whose values stay fixed (within this section).
We denote by $R$ the curvature tensor of $M$.
Since $R$ can be expressed in terms of sectional curvatures,
we have an estimate
\be\label{e:curvtensor}
 |R(u,v)w| \le C\cdot |u\wedge v|\cdot |w|
\ee
for any $u,v,w\in TM$ tangent to $M$ at a point where $|\Sec_M|\le 1$.

Let $z=\ga(L_1)$.
By construction and assumptions (1)--(3), the points $y$ and
$z$ belong to the ball $B_{1/5}(x)$,
$\inj_M(x)\ge 1$ and $|\Sec_M|\le 1$ on $B_1(x)$. 
For a smooth path $\sigma$ connecting $z$ to $y$,
let $P_\sigma\co T_yM\to T_zM$ denote the Levi-Civita parallel transport along $\sigma$.
Recall that $L(\sigma)$ denotes the length of a path $\sigma$.

%Consider unit tangent vectors
%$v_0=\dot\ga(L_1)\in S_zM$ and $v_1=\dot\ga_1(L_1)\in S_yM$.

\begin{lemma}\label{l:Psigmav0}
There exists a smooth path $\sigma$ connecting
connecting $z$ to $y$ and such that
\be\label{e:Lsigma}
L(\sigma)\le C_2\alpha
\ee
and
\be\label{e:Psigmav0}
|P_\sigma(\dot\ga(L_1))-\dot\ga_1(L_1)| \le C_2\alpha,
\ee
where $C_2=Ce^{CL_1}$ and $C$ is an absolute constant.
\end{lemma}

\begin{proof}
This lemma follows from the fact that, if $|\Sec_M|\le 1$,
then the time $L_1$ map of the geodesic flow
%(which is a self-diffeomorphism of $SM$
%sending $\dot\ga(0)$ to $v_0$ and $\dot\ga_1(0)$ to $v_1$)
is Lipschitz with Lipschitz constant $Ce^{CL_1}$.
To avoid discussion of the geodesic flow and the metric on~$SM$,
we give a direct proof of the lemma.

Connect the vectors $\dot\ga(0)$ and $\dot\ga_1(0)$ in $S_pM$
by a circle arc of length $\alpha$ and parametrize this arc as
$s\mapsto u(s)$, $s\in[0,\alpha]$, where $u(0)=\dot\ga(0)$, $u(\alpha)=\dot\ga_1(0)$,
and $|\frac d{ds} u(s)|=1$.
Consider the family of geodesics
$\ga_s\co[0,L_1]\to M$, $s\in[0,\alpha]$,
defined by the initial data $\ga_s(0)=p$ and $\dot\ga_s(0)=u(s)$.
Define $\sigma(s)=\gamma_s(L_1)$ and $V(s)=\dot\ga_s(L_1)$ for all $s\in[0,\alpha]$.
Note that the path $\sigma\co[0,\alpha]\to M$ connects $z$ to~$y$,
$V(0)=\dot\ga(L_1)$, and $V(\alpha)=\dot\ga_1(L_1)$.

Let $J_s$ be the Jacobi field along $\ga_s$ defined by
$J_s(t)=\frac d{ds} \ga_s(t)$.
Then
$$
 \frac \nabla{ds} V(s) = \frac \nabla{ds}  \frac d{dt}\Big|_{t=L_1} \ga_s(t)
 = \frac \nabla{dt}\Big|_{t=L_1} \frac d{ds} \ga_s(t) 
 = J_s'(L_1) .
$$
where $J_s'$ is a notation for $\frac\nabla{dt}J_s(t)$.
The Jacobi field $J_s(t)$ satisfies the Jacobi equation
$$
 J_s''= R(J_s,\dot\ga_s)\dot\ga_s,
$$
where $J_s''$ is a notation for $\frac\nabla{dt}\frac\nabla{dt}J_s$.
Hence
$ |J_s''| \le C\cdot |J_s| $
due to \eqref{e:curvtensor} and
the curvature bound $|\Sec_M|\le 1$ that we have everywhere on~$\ga_s$.
This and the initial data $J_s(0)=0$ and $|J_s'(0)|=|\frac d{ds} u(s)|=1$
imply that $|J_s(L_1)|$ and $|J_s'(L_1)|$ are bounded by $C_2=Ce^{CL_1}$
where $C$ is an absolute constant.
Thus
\be\label{e:Lsigma1}
 L(\sigma) = \int_0^\alpha | J_s(L_1) | \,ds \le C_2\alpha
\ee
and
\be\label{e:nablaV1}
 \int_0^\alpha \left| \frac\nabla{ds} V(s) \right| \,ds 
 = \int_0^\alpha \left| J'_s(L_1) \right| \,ds 
 \le C_2\alpha .
\ee
Note that \eqref{e:Lsigma1} implies \eqref{e:Lsigma}.
It remains to prove \eqref{e:Psigmav0}.

Let $W$ be the parallel vector field along $\sigma$
with initial data $W(0)=V(0)$.
Then
\be\label{e:nablaV2}
 |W(\alpha)-V(\alpha)|
 \le \int_0^\alpha \left| \frac\nabla{ds} (W(s) - V(s)) \right| \,ds
 = \int_0^\alpha \left| \frac\nabla{ds} V(s) \right| \,ds \le C_2\alpha
\ee
by \eqref{e:nablaV1}.
%since $\frac\nabla{ds} W(s)=0$.
Since $W(\alpha)=P_\sigma(V(0))=P_\sigma(\dot\ga(L_1))$ and
$V(\alpha)=\dot\ga_1(L_1)$,
the desired inequality  \eqref{e:Psigmav0} follows from \eqref{e:nablaV2}.
%This and \eqref{e:Lsigma1}
%imply the statement of the lemma.
\end{proof}

%The next lemma is another standard estimate for which the author
%could not find a reference.

\begin{lemma}\label{l:Psigma}
%Let $x,y,z$ be as above and
Let $\sigma_0,\sigma_1\co[0,1]\to B_1(x)$ be smooth paths connecting $z$ to~$y$.
Then, for every $v\in S_zM$,
\be\label{e:Psigma}
 | P_{\sigma_1}(v)- P_{\sigma_0}(v) | \le C (L(\sigma_1)+L(\sigma_2))^2
\ee
where $C$ is an absolute constant.
\end{lemma}

\begin{proof}
Let $\ell=L(\sigma_1)+L(\sigma_2)$.
%We may assume that $\ell\le 1/2$,
%otherwise \eqref{e:Psigma} trivially holds for any $C\ge 8$
%since the left-hand side of \eqref{e:Psigma} is bounded by~2.
%Then $\sigma_1$ and $\sigma_2$ are contained in the ball $B_1(x)$.
Due to our curvature and injectivity radius bounds,
the ball $B_1(x)$ is bi-Lipschitz equivalent, with an absolute bi-Lipschitz constant,
to a unit ball in $\R^n$ (see e.g.\ \cite[Ch.~5, Theorem 27]{Pe}).
Hence, by the Euclidean isoperimetric inequality,
the paths $\sigma_0$ and $\sigma_1$ span a 2-disc of area
no greater than~$C\ell^2$.
More precisely, $\sigma_0$ and $\sigma_1$ can be connected
by a smooth homotopy $\Sigma\co[0,1]\times[0,1]\to M$,
$\Sigma(0,\cdot)=\sigma_0$, $\Sigma(1,\cdot)=\sigma_1$,
with fixed endpoints $\Sigma(\cdot,0)=z$ and $\Sigma(\cdot,1)=y$,
and $\area(\Sigma)\le C\ell^2$.
Here $\area(\Sigma)$ is the surface area parametrized by $\Sigma$.
It can be expressed by the formula
$$
 \area(\Sigma) = \iint \left| \Sigma_s(s,t) \wedge \Sigma_t(s,t) \right| \,dsdt .
$$
where $\Sigma_s= \frac{\pd\Sigma}{\pd s}(s,t)$
and $\Sigma_t= \frac{\pd\Sigma}{\pd t}(s,t)$.

Fix $v\in S_zM$ and construct a smooth vector field $W=W(s,t)$ along $\Sigma$ as follows.
For every $s\in[0,1]$, let $W(s,\cdot)$ be the parallel
vector field along the path $\Sigma(s,\cdot)$
with initial data $W(s,0)=v$.
Then $ \frac\nabla{dt} W(s,t)=0 $ and therefore
\be\label{e:nabla2}
 \left|\frac{\nabla}{dt} \frac{\nabla}{ds} W\right|
 = \left|\frac{\nabla}{dt} \frac{\nabla}{ds} W - \frac{\nabla}{ds} \frac{\nabla}{dt} W\right|
 = |R(\Sigma_t,\Sigma_s) W|
\le C |\Sigma_t \wedge \Sigma_s |
\ee
where we have omitted the arguments $(s,t)$.
The last inequality follows from \eqref{e:curvtensor} and the fact that $|W(s,t)|=|v|=1$.
Since $W(\cdot,0)$ is constant, $\frac{\nabla}{ds} W$ vanishes at $t=0$.
Hence, by integration, \eqref{e:nabla2} implies that
$$
\left|\frac{\nabla}{ds} W(s,1)\right| 
 \le \int_0^1 \left| \frac{\nabla}{dt} \frac{\nabla}{ds} W \right| dt
 \le C \int_0^1 |\Sigma_t \wedge \Sigma_s | \, dt .
$$
Recall that $W(s,1)$ belong to one tangent space $T_yM$ for all $s$.
Now integration with respect to $s$ yields that
$$
 | W(1,1)-W(0,1) | \le C \iint \left| \Sigma_s \wedge \Sigma_t \right| \,dsdt  = C \area(\Sigma) \le C\ell^2.
$$
Since $W(1,1)=P_{\sigma_1}(v)$ and $W(0,1)=P_{\sigma_0}(v)$, the lemma follows.
\end{proof}

Now we are in a position to prove Lemma~\ref{l:beta-vs-alpha}.
Let $\sigma$ be a path constructed in Lemma \ref{l:Psigmav0}.
We may assume that $\alpha<1/(2C_2)$, otherwise \eqref{e:beta-goal}
holds for any $C_1\ge 10C_2$ since $\beta\le\pi$.
Then \eqref{e:Lsigma} implies that $L(\sigma)<1/2$, hence $\sigma$
is contained in the ball $B_1(x)$.

We apply Lemma \ref{l:Psigma} to the paths $\sigma_0=[zy]$ and $\sigma_1=\sigma$,
both reparametrized by $[0,1]$, and $v=\dot\ga(L_1)$.
Since $\sigma_0$ is a reparametrization of $\ga|_{[L_1,L+r]}$,
we have $P_{\sigma_0}(\dot\ga(L_1))=\dot\ga(L+r)$.
Thus, by Lemma \ref{l:Psigma},
$$
 |\dot\ga(L+r) - P_\sigma(\dot\ga(L_1))| \le C(L(\sigma)+|zy|)^2 \le 4C L(\sigma)^2 \le 2C L(\sigma)
$$
since $|zy|\le L(\sigma)\le 1/2$.
This and \eqref{e:Lsigma} imply that
$$
 |\dot\ga(L+r) - P_\sigma(\dot\ga(L_1))| \le C_3\alpha
$$
for a suitable constant $C_3=C_3(R_0)>0$.
This and \eqref{e:Psigmav0} imply that
\be\label{e:sasaki1}
 |\dot\ga(L+r) - \dot\ga_1(L_1)| \le (C_2+C_3)\alpha .
\ee
On the other hand,
$|\dot\ga(L+r) - \dot\ga_1(L_1)| = 2\sin\frac\beta2 \ge 2\beta/\pi$.
This and \eqref{e:sasaki1} imply the estimate
\eqref{e:beta-goal} with $C_1=\frac\pi2(C_2+C_3)$.
This finishes the proof of Lemma~\ref{l:beta-vs-alpha}.

\section{Invertibility of $\D_F$}
\label{sec:inverse}

In this section we prove Theorem \ref{t:inverse} and Corollary \ref{cor:bilip}.
The proof consists of two steps.
First, in Proposition \ref{p:holder} we show that $\D_F$ is a homeomorphism
onto its image and obtain a preliminary H\"older estimate on $\D_F^{-1}$.
Second, in Proposition \ref{p:bilip} we show that $\D_F$ is locally bi-Lipschitz.
The two propositions are combined in subsection \ref{subsec:proof-inverse}
into a proof of Theorem \ref{t:inverse} and Corollary \ref{cor:bilip}.

Let $M$ and $F$ be as in Theorem \ref{t:inverse}.
By rescaling we may assume that $F$ contains a ball of
radius~1.
The center of this ball is denoted by $q_0$.
We also assume that this ball $B_1(q_0)$ is not the whole $M$.
In particular this implies that $\diam(M)\ge 1$.

\subsection{Preliminary H\"older estimate}
First we show that $\D_F$ is a homeomorphism onto
its image $\D_F(M)\subset\C(F\times F)$.
%The triangle inequality implies that 
Since $\D_F$ is a 2-Lipschitz map,
we only need to prove
that $\D_F$ is injective and its inverse map is continuous.
This follows from the next proposition,
which also shows that $\D_F^{-1}$ 
is locally H\"older continuous with exponent~$1/2$.
%and local H\"older norm determined by geometric bounds.

\begin{proposition}\label{p:holder}
For every $n\ge 2$ and $K,R,\nu>0$ there exist $C_0>0$ and $\de_0>0$
such that the following holds.
Let $M^n$ be a complete Riemannian manifold, $F\subset M$,
and assume that $F$ contains a unit ball $B_1(q_0)\subsetneq M$
with $\vol(B_1(q_0))\ge \nu$
and that $\Sec_M\ge -K$ everywhere in $B_{3R+3}(q_0)$.

Then for every $x\in B_R(q_0)$ and $y\in M$ such that
$$
  \|\D_F(x)-\D_F(y)\|<\de_0,
$$
one has
\be\label{e:holder}
 |xy| \le C_0 \|\D_F(x)-\D_F(y)\|^{1/2} .
\ee
\end{proposition}

\begin{remark}
Proposition \ref{p:holder} requires only a lower curvature bound
and its proof can be modified to work for any finite-dimensional
Alexandrov space of curvature bounded below.
The bound $\vol(B_1(q_0))\ge\nu$
is a non-collapsing assumption.
If $M$ has two-sided curvature bounds, then $\vol(B_1(q_0))$
is bounded away from 0 in terms of the curvature bounds and the injectivity radius $\inj_M(q_0)$,
see \cite[Ch.~6, Theorem 27]{Pe}.
\end{remark}

\begin{proof}[Proof of Proposition \ref{p:holder}]
Let $M,F,q_0$ be as in the proposition,
$x\in B_R(q_0)$ and $y\in M\setminus\{x\}$.
Let $\de=\|\D_F(x)-\D_F(y)\|$
and assume that $\de<\de_0$
where $\de_0$ is sufficiently small, depending on the parameters
$n,K,R,\nu$.
The required bounds on $\de_0$ are accumulated in the course of the proof.
We denote by $C_1$, $c_1$, etc, various positive constants
determined by $n,K,R,\nu$.

We may assume that $|q_0x|\le|q_0y|$, otherwise swap $x$ and~$y$.
Let $a=|q_0y|-|q_0x|$.
For every $p\in F$ we have
$$
 \bigl| |py|-|px|-a \bigr| = | D_y(p,q_0)-D_x(p,q_0) | \le \|\D_F(y)-\D_F(x)\| \le \de .
$$
Therefore
\be\label{e:py-px}
 |py| \ge |px| + a - \de \ge |px| - \de
\ee
for all $p\in F$.

\begin{lemma}\label{l:alpha0}
There exists $\alpha_0=\alpha_0(n,K,R,\nu)$ such that the following holds.
For any $M,q_0,x,y$ as above,
there is a point $q\in B_{1/2}(q_0)$ such that $|qx|\ge1/10$
and 
$%\be\label{e:alpha0}
\angle qxy\le \pi-\alpha_0
$ %\ee
(for some choice of minimizing geodesics $[xy]$ and $[xq]$
defining this angle).
%where $\alpha_0>0$ is determined by $n,K,R,\nu$.
\end{lemma}

\begin{proof}
Since $\diam(M)\ge 1$, there is a point $q_1\in M$ such that $|q_0q_1|\le 2/5$
and $|xq_1|\ge 1/5$.
Indeed, if $|q_0x|\ge1/5$ then one can take $q_1=q_0$, 
otherwise let $q_1$ be any point such that $|q_0q_1|=2/5$.
We are going to choose $q$ from the ball $B:=B_{1/10}(q_1)$.
This guarantees that $q\in B_{1/2}(q_0)$ and $|qx|\ge 1/10$.

By the Bishop-Gromov inequality (see \cite[Ch.~9, Lemma 36]{Pe}),
we have a lower bound on the volume ratio
$$
  \frac{\vol(B)}{\vol(B_{2}(q_1))} \ge c_1 = c_1(n,K) > 0.
$$
Therefore
\be\label{e:volB}
 \vol(B) \ge c_1 \vol(B_{2}(q_1))
 \ge c_1 \vol(B_{1}(q_0)) \ge c_1\nu .
\ee
This and Lemma \ref{l:dirvol} imply a lower bound
\be\label{e:c_2}
 \Hm^{n-1}(\dir(x,B)) \ge c_2 = c_2(n,K,R,\nu) > 0 .
\ee
Here we use the assumptions that $x\in B_R(q_0)$ and
$\Sec_M\ge -K$ on $B_{3R+3}(q_0)$.
They imply that $B\subset B_{R+1}(x)$ and 
$\Sec_M\ge -K$ on $B_{2R+2}(x)$.
Hence Lemma \ref{l:dirvol}, \eqref{e:volB}, and the fact that $\diam(B)<1$
imply \eqref{e:c_2} for $c_2=\la_{K,R+1}^{n-1}c_1\nu$
where $\la_{K,R+1}$ is the constant from Lemma \ref{l:dirvol}.

Fix a minimizing geodesic $[xy]$ and let $v\in S_xM$
be its direction at~$x$.
For $\alpha>0$ let $\Sigma_\alpha$ denote the set of
all $u\in S_xM$ such that $\angle(u,v)>\pi-\alpha$.
In other words, $\Sigma_\alpha$ is the $\alpha$-neighborhood
of $-v$ in the unit sphere $S_xM$.
The volume $\Hm^{n-1}(\Sigma_\alpha)$ of this neighborhood
goes to 0 as $\alpha\to 0$.
Hence for a sufficiently small $\alpha_0=\alpha_0(n,c_2)>0$
we have $\Hm^{n-1}(\Sigma_{\alpha_0})< c_2$.
This and \eqref{e:c_2} imply that that there exists $u\in\dir(x,B)$
such that $u\notin\Sigma_{\alpha_0}$ and hence $\angle(u,v)\le\pi-\alpha_0$.
By the definition of $\dir(x,B)$, the vector $u$ is the initial direction
of a minimizing geodesic $[xq]$ for some $q\in B$, 
hence the result.
\end{proof}

Fix $q\in B_{1/2}(q_0)$
and minimizing geodesics $[xq]$ and $[xy]$
constructed in Lemma~\ref{l:alpha0}.
We define
$$
 r_1=\min\{1/30,K^{-1/2}\}
$$
and require that $\de_0\le r_1$.
The curvature bound assumed in Proposition \ref{p:holder}
implies a bound on the lower curvature radius (see Definition \ref{d:lcr}):
\be\label{e:lcr-bound}
 \lcr(z)\ge r_1 \quad\text{for all $z\in B_{R+2}(q_0)$}
\ee

Recall that $|qx|\ge 1/10 \ge 3r_1$.
Hence there exists $p\in [qx]$ such that $|pq|=r_1$.
Note that
\be\label{e:px-bounds}
  2r_1 \le  |px| \le R+1
\ee
and
\be\label{e:py-ge-r1}
 |py| \ge |px|-\de \ge r_1
\ee
due to \eqref{e:py-px}, \eqref{e:px-bounds} 
and the requirement that $\de<\de_0\le r_1$.
Define
$$
\beta = \angle xpy = \pi-\angle qpy .
$$
%Our next goal is to show that $\beta$ is small.

\begin{lemma}\label{l:beta}
$\beta\le C_2\de^{1/2}$ for some $C_2>0$ determined by $n,K,R,\nu$.
\end{lemma}

\begin{proof}
Both $p$ and $q$ belong to the ball $B_1(q_0)$ and hence to~$F$.
Therefore
$$
 D_x(q,p)-D_y(q,p) \le \|\D_F(x)-\D_F(x)\| = \de .
$$
Substituting $D_x(q,p) = |qx|-|px| = |pq|$
and $D_y(q,p)=|qy|-|py|$ yields that
\be\label{e:pqy-shortcut}
  |pq|+|py|-|qy| \le \de .
\ee
On the other hand, by Lemma \ref{l:shortcut},
\be\label{e:pqy-shortcut2}
 |pq|+|py|-|qy| \ge c\beta^2\min\{|pq|,|py|,\lcr(p)\} = c\beta^2 r_1 ,
\ee
where $c>0$ is an absolute constant.
The last identity in \eqref{e:pqy-shortcut2} follows from \eqref{e:lcr-bound}, 
\eqref{e:py-ge-r1}, and the construction of~$p$.
Now \eqref{e:pqy-shortcut} and \eqref{e:pqy-shortcut2} imply that
$\beta \le C_2\de^{1/2}$ where $C_2=(cr_1)^{-1/2}$.
\end{proof}

By \eqref{e:py-ge-r1},
there is a point $x_1\in [py]$ such that
\be\label{e:px1-ge-r1}
 |px_1|  = |px|-\de \ge r_1 .
\ee
By \eqref{e:px-bounds}, we have  $|px_1|\le |px| \le R+1$.
The curvature bounds assumed in Proposition \ref{p:holder} imply that
$\Sec_M\ge-K$ everywhere in the ball $B_{2R+2}(p)$.
Hence, by Lemma \ref{l:lipexp} applied to $p,x,x_1$,
\be\label{e:xx_1}
  |xx_1| \le C_3\beta + \de \le C_4 \de^{1/2}
\ee
where $C_3$ is the constant $C_{K,R+1}$ from Lemma \ref{l:lipexp}
and $C_4=C_3C_2+r_1^{1/2}$. 
The last inequality in \eqref{e:xx_1} follows from Lemma \ref{l:beta}
and the assumption that $\de<\de_0\le r_1$.

We may assume that 
\be\label{e:xy-lower}
|xy|>2C_4\de^{1/2} ,
\ee
otherwise \eqref{e:holder} holds for any $C_0\ge 2C_4$.
By \eqref{e:xx_1} and \eqref{e:xy-lower},
$$
  |x_1y|\ge |xy|-|xx_1| > \frac12 |xy| . %>C_4\de^{1/2}
$$
This inequality, \eqref{e:lcr-bound} and \eqref{e:px1-ge-r1} imply that
$$
  \min\{|px_1|,|x_1y|,\lcr(x_1)\} \ge \frac12 \min\{ |xy|,r_1 \} . %\ge C_4\de^{1/2} .
$$
Now we apply the second part of Lemma \ref{l:short-median} to points $p$, $y$, $x_1$, $x$
and obtain that
\be\label{e:pxy-excess-above}
 |px| + |xy| - |py| \le  \frac{4C |xx_1|^2}{\min\{ |xy|, r_1 \}}
 \le \frac{C_5 \de}{\min\{ |xy|, r_1 \}}
\ee
where the second inequality follows from \eqref{e:xx_1}.
Here $C$ is the absolute constant from Lemma \ref{l:short-median}
and $C_5=4CC_4^2$.

On the other hand, 
recall that $\angle pxy=\angle qxy\le\pi-\alpha_0$ by Lemma \ref{l:alpha0}.
This and Lemma \ref{l:shortcut} imply that
$$%\be\label{e:pxy-excess-below}
 |px| + |xy| - |py| \ge c\alpha_0^2 \min\{ |px|, |xy|, \lcr(x) \}
 \ge c\alpha_0^2 \min\{ |xy|, r_1 \}
$$
where the second inequality follows from \eqref{e:lcr-bound} and \eqref{e:px-bounds}.
This and \eqref{e:pxy-excess-above} imply
\be\label{e:holder-almost-final}
 \min\{ |xy|, r_1 \}^2 \le C_6 \de
\ee
where $C_6=c^{-1} C_5 \alpha_0^{-2}$.
We require that $\de_0< C_6^{-1}r_1^2$,
then $r_1^2 \ge C_6\de_0 > C_6\de$.
This and \eqref{e:holder-almost-final} imply that $|xy|\le C_6^{1/2}$,
thus \eqref{e:holder} holds for any $C_0\ge C_6^{1/2}$.

Collecting all estimates obtained above, we conclude that \eqref{e:holder}
holds for
$$
  C_0 := \max\{ 2C_4, C_6^{1/2} \}
$$
provided that
$$
 \de < \de_0 := \min \{ r_1, C_6^{-1}r_1^2 \}  .
$$
%
%It remains to prove the last statement of Proposition \ref{p:holder}.
%Since $F$ has nonempty interior, one can fix a ball contained in~$F$
%and let $q_0$ be its center, $\rho$ its radius and $\nu$ its volume.
%For every $x\in M$ there exist parameters $R$ and $K$
%such that the assumptions of Proposition \ref{p:holder} are satisfied,
%moreover they can be chosen uniformly over bounded subsets of $M$.
%Then \eqref{e:holder} implies that $\D_F$ is injective
%and $\D_F^{-1}$ is H\"older continuous with exponent $1/2$.
Since all constants in the argument are determined by $n,K,R,\nu$,
this proves Proposition~\ref{p:holder}.
\end{proof}

\subsection{Local bi-Lipschitz continuity}
Proposition~\ref{p:holder} implies that
$\D_F$ is a homeomorphism onto its image.
Thus, in order to prove Theorem \ref{t:inverse}
it suffices to verify that $\D_F$ is locally bi-Lipschitz.
This is established in the following proposition.

\begin{proposition}\label{p:bilip}
For every $n\ge 2$ and $K,R,i_0>0$ there exist $c_0>0$ and $r_0>0$
such that the following holds.
Let $M^n$ be a complete Riemannian manifold, $F\subset M$,
and assume that $F$ contains a unit ball $B_1(q_0)\subsetneq M$,
and that $|\Sec_M|\le K$ and $\inj_M\ge i_0$
everywhere in the ball $B_{3R+3}(q_0)$.

Then for all $x,y\in B_R(q_0)$ such that $|xy|<r_0$,
one has
\be\label{e:bilip}
 \|\D_F(x)-\D_F(y)\| \ge c_0 |xy| .
\ee
\end{proposition}

\begin{proof}
Let $M,F,q_0$ be as in the proposition.
Fix $x,y\in B_R(q_0)$.
We begin with a lemma similar to Lemma \ref{l:alpha0}.

\begin{lemma}\label{l:theta0}
There exists $\theta_0=\theta_0(n,K,R,i_0)>0$ such that the following holds.
For any $M,q_0,x,y$ as above, 
there exist $q_1,q_2\in B_{1/2}(q_0)$ such that $|q_1x|\ge1/10$,
$|q_2x|\ge1/10$, and 
\be\label{e:theta0}
 |\cos\angle q_1xy- \cos\angle q_2xy| \ge \theta_0
\ee
for some choice of minimizing geodesics $[xy]$, $[xq_1]$ and $[xq_2]$
defining the above angles.
\end{lemma}

\begin{proof}
As in the proof of Lemma \ref{l:alpha0}, 
fix $q_1\in M$ such that $|q_0q_1|\le 2/5$ and $|xq_1|\ge 1/5$,
and let $B=B_{1/10}(q_1)$.
The bounds on curvature and injectivity radius imply a lower bound
on the volume of $B$:
$$
  \vol(B) \ge c_1=c_1(n,K,i_0) > 0 .
$$
Similarly to \eqref{e:c_2}, this and Lemma \ref{l:dirvol} imply a lower bound
$$
 \Hm^{n-1}(\dir(x,B)) \ge c_2 = c_2(n,K,R,i_0) > 0 .
$$
Fix minimizing geodesics $[xy]$ and $[xq_1]$,
let $v$ and $u_1$ be their directions at~$x$
and $\alpha=\angle(v,u_1)$.
For $\theta>0$, consider the set
$$
\Sigma_\theta = \{ u\in S_xM : |\cos\angle(u,v)-\cos\alpha|<\theta \}
$$
in the unit sphere $S_xM$.
Observe that $\Hm^{n-1}(\Sigma_\theta)\to 0$ as $\theta\to 0$ uniformly in $\alpha$.
%is bounded above
%by $f_n(\theta)$ where $f_n$ is a function determined by~$n$
%and $f_n(\theta)\to 0$ as $\theta\to 0$.
Hence there exists $\theta_0=\theta_0(n,c_2)$ such that
$\dir(x,B)$ contains a vector $u_2\notin\Sigma_{\theta_0}$.
This vector is the initial direction of a minimizing geodesic
$[xq_2]$ where $q_2\in B$.
By the definition of $\Sigma_{\theta_0}$,
\eqref{e:theta0} holds for $q_1$, $q_2$ and the minimizing geodesics
$[xy]$, $[xq_1]$, $[xq_2]$ fixed above.
\end{proof}

Fix $q_1,q_2$ and minimizing geodesics $[xy]$, $[xq_1]$ and $[xq_2]$
constricted in Lemma~\ref{l:theta0}.
For each $i=1,2$,
let $p_i$ be the point on $[xq_i]$ such that $|q_ip_i|=1/20$.
Note that $|p_ix|=|q_ix|-|q_ip_i|\ge 1/20$.
%By Proposition \ref{p:extension}, there exists a minimizing extension
%$[p_iz_i]$ of $[p_ix]$ such that $|xz_i|\ge r_1>0$
%where $r_1$ is determined by $K,R,i_0$.
%Let
%$$
% r_2 = \min\{1/20,r_1, K^{-1/2} \} .
%$$
We apply Corollary \ref{cor:1stvar} to $q_i,p_i,x,y$ and obtain
$$
 \bigl| |p_ix|-|p_iy| - |xy| \cos\alpha_i \bigr| \le C_7 |xy|^2 , \qquad i=1,2,
$$
where $\alpha_i=\angle p_ixy=\angle q_ixy$ and $C_7$ is
a constant determined by $K,R,i_0$.
Hence
\begin{multline*}
 |D_x(p_1,p_2)-D_y(p_1,p_2)| 
 = \bigl| (|p_1x|-|p_1y|) - (|p_2x|-|p_2y|) \bigr| \\
 \ge |xy|\cdot|\cos\alpha_1-\cos\alpha_2| - 2C_7 |xy|^2
  \ge \theta_0|xy| - 2C_7 |xy|^2
\end{multline*}
where the last inequality follows from Lemma \ref{l:theta0}.
Since both $p_1$ and $p_2$ belong to to $B_1(q_0)$ and hence to $F$,
this implies
\be\label{e:bilip1}
 \|\D_F(x)-\D_F(y)\| \ge \theta_0|xy| - 2C_7 |xy|^2 .
\ee
If $|xy|<r_0:=\theta_0/4C_7$, then the right-hand side of \eqref{e:bilip1}
is bounded below by $\theta_0|xy|/2$
and hence \eqref{e:bilip} holds for $c_0=\theta_0/2$.
This finishes the proof of Proposition~\ref{p:bilip}.
\end{proof}

\subsection{Proof of Theorem \ref{t:inverse} and Corollary \ref{cor:bilip}}
\label{subsec:proof-inverse}
Now we deduce Theorem \ref{t:inverse} and Corollary \ref{cor:bilip}
from Propositions \ref{p:holder} and \ref{p:bilip}.
As explained in the beginning of this section, we rescale $M$
so that, upon the rescaling, $F$ contains a unit ball $B_1(q_0)$ and $B_1(q_0)\ne M$.
Note that in the case of Corollary \ref{cor:bilip}
one can use the rescaling factor equal to $\min\{\rho_0,i_0,1\}^{-1}$.
%determined by the given geometric bounds.

Fix $x_0\in M$ and choose geometric bounds $K,R,i_0,\nu$ such that
Propositions \ref{p:holder} and \ref{p:bilip} apply to any $x\in B_1(x_0)$.
Namely pick any $R\ge |q_0x_0|+1$, 
let $K$ be the supremum of $|\Sec_M|$ over the ball $B_{3R+3}(q_0)$,
let $i_0$ be the infimum of $\inj_M$ over the same ball,
and let $\nu$ be the volume of the ball of radius $\min\{i_0,1\}$ in
the $n$-sphere equipped with a metric of constant curvature~$K$.

Proposition \ref{p:bilip} provides 
positive numbers $r_0$ and $c_0$
such that \eqref{e:bilip} holds for all pairs $x,y\in B_1(x_0)$ such that $|xy|<r_0$.
We may assume that $r_0\le 1$.

Proposition \ref{p:holder} implies that there exists $\de_1>0$ such that
$|xy|<r_0$ for all  $x\in B_1(x_0)$ and $y\in M$ such that
$ \|\D_F(x)-\D_F(y)\| < 2\de_1 $.
In particular, if $\|\D_F(x)-\D_F(x_0)\| < \de_1$
and $\|\D_F(y)-\D_F(x_0)\| < \de_1$,
then both $x$ and $y$ belong to 
the ball $B_{r_0}(x_0)\subset B_1(x_0)\subset B_R(q_0)$ and $|xy|<r_0$.
This and \eqref{e:bilip} imply that $\D_F^{-1}$ is $c_0^{-1}$-Lipschitz
within the ball of radius $\de_1$ centered at $\D_F(x_0)$.
This finishes the proof of Theorem \ref{t:inverse}.

To prove Corollary \ref{cor:bilip}, observe that in the case of compact $M$
the parameters above can be chosen independently of $x_0$
and uniformly for all $M\in\mathcal M(n,D,K,i_0)$.
The above argument shows that, if $F$ contains a unit ball $B_1(q_0)\ne M$,
then $\D_F^{-1}$ has a global Lipschitz
constant $\max\{c_0^{-1},\de_1^{-1}\diam(M)\}$,
which is determined by $n,D,K,i_0$.
%Removing the unit ball assumption by means of
The initial rescaling adds the dependence on $\rho_0$.

\section{Determination of the metric}
\label{sec:reconstruction}

In this section we prove Theorem \ref{t:reconstruction}.
Let $M_i$, $F$, $U_i$, $\D_F^i$, $i=1,2$, 
be as in Theorem \ref{t:reconstruction}.
We denote the distances in $M_1$ and $M_2$ by $d_1$ and $d_2$.

By Theorem \ref{t:inverse}, $\D_F^i|_{U_i}$ 
is a locally bi-Lipschitz homeomorphism between $U_i$ and its image $\D_F^i(U_i)$.
Hence the map $\phi\co U_1\to U_2$ given by \eqref{e:def-phi}
is well-defined and it is a locally bi-Lipschitz homeomorphism between $U_1$ and~$U_2$.
Our goal is to prove that $\phi$ is a Riemannian isometry.
By the definition of $\phi$ we have
\be\label{e:phi}
  \D_F^1(x) = \D_F^2(\phi(x))
\ee
for all $x\in U_1$.

First we show that the distance difference data determines
the metric of the observation domain $F$.
%Namely the following holds.

\begin{lemma}\label{l:g-on-F}
$g_1|_F=g_2|_F$.
\end{lemma}

\begin{proof}
Fix a point $q_0\in F$.
For $i\in\{1,2\}$ and $x\in U_i$ define a function
$f_x^i\co F\to\R$ by
$$
 f_x^i(y) = \D^i_F(x)(y,q_0) = d_i(x,y) - d_i(x,q_0) , \qquad y\in F.
$$
By \eqref{e:phi} we have $f^1_x=f^2_{\phi(x)}$ for all $x\in U_1$.
Hence the two sets $\{f^1_x\}_{x\in U_1}$ and $\{f^2_x\}_{x\in U_2}$ coincide.

If the function $f_x^i$ is differentiable at $y\in F$,
then $x$ is connected to $y$ by a unique minimizing geodesics of $M_i$,
and the velocity of this geodesic at $y$ %(directed outwards the geodesic)
is the Riemannian
gradient of $f^i_x$ at~$y$ with respect to~$g_i$.
In particular the derivative $d_yf^i_x$ is a co-vector of unit norm with respect to $g_i$.
Thus
\be\label{e:unit-covector}
 \|d_yf^1_x\|_{g_1}= 1=\|d_yf^2_{\phi(x)}\|_{g_2} = \|d_yf^1_x\|_{g_2}
\ee
whenever the function $f^1_x=f^2_{\phi(x)}$ is differentiable at~$y$.

For every $y\in F$ and every vector $v\in\dir_{M_1}(y,F)$
(see Notation \ref{n:dir})
there exists $x\in U_1$
such that $f^1_x$ is differentiable at~$y$ and the $g_1$-gradient
of $f^1_x$ at $y$ equals $-v$.
To construct such $x$, pick $x_0\in U_i$ such that
$v$ is the direction of a minimizing geodesic $[yx_0]$ at $y$,
and let $x$ be any point from $([yx_0]\cap U_i)\setminus\{x_0\}$.
By Lemma \ref{l:dirvol} it follows that the set of co-vectors
$$
 \{ d_yf^1_x : x\in U_1 \text{ and $f^1_x$ is differentiable at $y$} \} \subset T_y^*M_1 
$$
has a positive
$(n-1)$-dimensional Hausdorff measure.
By \eqref{e:unit-covector}, all these co-vectors have unit norms
with respect to both $g_1$ and $g_2$. Thus the Euclidean norms
defined by $g_1$ and $g_2$ on the co-tangent space at $y$
agree on a set of positive measure.
Hence these norms coincide and so do the metric tensors $g_1$ and $g_2$ at~$y$.
Since $y\in F$ is arbitrary, the lemma follows.
\end{proof}

For any subset $F'\subset F$,
the distance difference data $\D^i_F(U_i)$ determines $\D^i_{F'}(U_i)$,
hence it suffices to prove the theorem for $F'$ in place of~$F$.
More precisely, let $F'\subset F$ be a set with a nonempty interior.
Then the distance difference representations $\D_{F'}^i\co M_i\to\C(F'\times F')$ are injective
by Theorem \ref{t:inverse}, $\D_{F'}^1(U_1)=\D_{F'}^2(U_2)$ since
$\D_{F}^1(U_1)=\D_{F}^2(U_2)$, and 
$\D_{F'}^1(x) = \D_{F'}^2(\phi(x))$ for all $x\in U_1$ due to \eqref{e:phi}.
These properties imply that
\be\label{e:phiprime}
 \phi = (\D^2_{F'})^{-1}\circ\D^1_{F'}|_{U_1}
\ee
(compare with \eqref{e:def-phi}),
therefore Theorem \ref{t:reconstruction} follows from a similar statement
with $F'$ in place of~$F$.

In particular, we may replace $F$ by a small geodesic ball of the metric $g_1$.
By Lemma \ref{l:g-on-F}, this ball is a ball of $g_2$ as well.
Since small geodesic balls are convex,
this and Lemma \ref{l:g-on-F} imply that
\be \label{e:diF}
d_1|_{F\times F}=d_2|_{F\times F} .
\ee
In the sequel we assume that $F$ is a convex geodesic ball
in both $M_1$ and $M_2$ and hence \eqref{e:diF} holds.

Our next step is to handle the intersection $F\cap U_i$.

\begin{lemma}\label{l:UF}
$U_1\cap F=U_2\cap F$ and $\phi|_{U_1\cap F}$ is the identity map.
\end{lemma}

\begin{proof}
Fix $q\in F$ and let $x\in U_1\cap F$.
Consider a function $f\co F\to\R$ defined by
$$
 f(y) = \D_F^1(x)(y,q) = d_1(x,y)-d_1(x,q)
$$
and observe that it attains its minimum at~$x$.
Let $\tilde x=\phi(x)$.
By \eqref{e:phi} we have
\be\label{e:UF1}
 f(y) = \D_F^2(\phi(x))(y,q) = d_2(\tilde x,y)-d_2(\tilde x,q)
\ee
for all $y\in F$.
Suppose that $\tilde x\ne x$. Consider a minimizing geodesic $\ga$ of $M_2$
connecting $x$ and $\tilde x$ and pick a point $y\in\ga\cap F\setminus\{x\}$.
For this point $y$ we have $d_2(\tilde x,y)<d_2(\tilde x,x)$ and
hence, by \eqref{e:UF1}, $f(y)<f(x)$.
This contradicts the fact that $f(x)$ is the minimum of~$f$,
therefore $\tilde x=x$.

Thus $\phi|_{U_1\cap F}$ is the identity, in particular $U_1\cap F\subset U_2$.
Switching the roles of $M_1$ and $M_2$ shows that $U_1\cap F=U_2\cap F$.
\end{proof}

Lemma \ref{l:UF} allows us to assume that $F\cap U_i=\emptyset$
%and moreover $F$ and $U_i$ are separated by a positive distance 
for each $i\in\{1,2\}$.
Indeed, by Lemma \ref{l:UF} and Lemma \ref{l:g-on-F} the restriction $\phi|_{U_1\cap F}$ is
a Riemannian isometry, hence it suffices to prove that the restriction of
$\phi$ to any open set containing $U_1\setminus F$ is an isometry as well.
Recall that $F$ is a geodesic ball (with respect to both $g_1$ and~$g_2$),
i.e.\ $F=B_{3\rho}(q_0)$ for some $q_0\in F$ and $\rho>0$.
Thus we may replace each $U_i$ by $U_i'=U_i\setminus \overline B_{2\rho}(q_0)$
and prove the theorem for $U_i'$ in place of $U_i$, $i=1,2$.
Then replacing $F$ by $F'=B_{\rho}(q_0)$ does not change the map $\phi$, see \eqref{e:phiprime},
hence it suffices to prove the theorem for $F'$ and $U_i'$ in place of $F$ and~$U_i$.
We reuse the notation $F$ and $U_i$ for $F'$ and $U_i'$ and thus
assume that $F\cap U_i=\emptyset$.
%Moreover we may assume that $F$ is separated from each $U_i$ by a positive distance~$\rho$
%in the respective Riemannian manifold $(M_i,g_i)$.

The next lemma tells that $\phi$ maps certain unparametrized geodesics of $U_1$
to unparametrized geodesics of $U_2$.
This is similar to \cite[Lemma 2.9]{LS}.
We give a different proof since the argument in \cite{LS}
uses more regularity assumptions than those guaranteed by curvature bounds.

\begin{lemma}\label{l:dirphi}
Let $x\in U_1$ and let $\ga$ be a minimizing geodesic of $M_1$
connecting $x$ to $q\in F$.
Then there exists a minimizing
geodesic $\tilde\ga$ of $M_2$ connecting $\phi(x)$ to~$q$
and such that $\phi(\ga\cap U_1)=\tilde\ga\cap U_2$ and $\ga\cap F=\tilde\ga\cap F$.
\end{lemma}

\begin{proof}
Let $\tilde x=\phi(x)$.
Since $g_1|_F=g_2|_F$ and $F$ is convex in both $M_1$ and $M_2$,
$\ga\cap F$ is a minimizing geodesic arc in $M_2$ as well as in $M_1$.
First we show that a suitable extension of this arc in $M_2$
is a minimizing geodesic connecting $\tilde x$ to~$q$.

Fix $p\in\ga\cap F\setminus\{q\}$. Since $\ga$ is a minimizing geodesic of $M_1$, we have
$$
 \D^1_F(x)(q,p) = d_1(x,q)-d_1(x,p) = d_1(q,p) .
$$
This, \eqref{e:phi}, and \eqref{e:diF} imply that
$\D^2_F(\tilde x)(q,p) = d_2(q,p)$, or, equivalently
$$
 d_2(q,\tilde x)=d_2(q,p)+d_2(p,\tilde x) .
$$
This implies that there is a minimizing geodesic of $M_2$ connecting $\tilde x$ to $q$
and containing~$p$.
Denote this geodesic by $\tilde\ga$.
By the convexity of $F$ we have $\tilde\ga\cap F=\ga\cap F$.
It remains to prove that $\phi(\ga\cap U_1)=\tilde\ga\cap U_2$.

Let $y\in\ga\cap U_1$ and $\tilde y=\phi(y)$. The same argument applied to $y$ in place of $x$
shows that $\tilde y$ belongs to either $\tilde\ga$ or to a minimizing extension
of $\tilde\ga$ beyond $\tilde x$.
Let us show that $\tilde y$ belongs to $\tilde\ga$ rather than an extension.

Since $\dim M_1\ge 2$, there is a point $z\in F$ 
satisfying a strict triangle inequality
$$
d_1(x,z)< d_1(x,y)+d_1(y,z) .
$$
Subtracting the identity $d_1(x,p)=d_1(x,y)+d_1(y,p)$ yields that
$$
 \D_F^1(x)(z,p) < \D_F^1(y)(z,p)
$$
and hence, by \eqref{e:phi},
\be\label{e:dirphi1}
 \D_F^2(\tilde x)(z,p) < \D_F^2(\tilde y)(z,p)
\ee
On the other hand, if $\tilde y$ belongs to an extension of $\tilde\ga$ beyond $\tilde x$,
then 
$$
d_2(\tilde y,p)=d_2(\tilde y,\tilde x)+d_2(\tilde x,p) .
$$
Subtracting this from the triangle inequality
$d_2(\tilde y,z)\le d_2(\tilde y,\tilde x)+d_2(\tilde x,z)$
yields that
$
\D^2_F(\tilde y)(z,p)\le \D^2_F(\tilde x)(z,p)
$,
contrary to \eqref{e:dirphi1}.
Therefore $\tilde y$ belongs to $\tilde\ga$.

This proves that $\phi(\ga\cap U_1)\subset\tilde\ga$.
Switching the roles of $M_1$ and $M_2$ yields that 
$\phi^{-1}(\tilde\ga\cap U_2)\subset\ga$.
Hence $\phi(\ga\cap U_1)=\tilde\ga\cap U_2$
and the lemma follows.
\end{proof}

Now we begin the main part of the proof of Theorem \ref{t:reconstruction}.
Fix $x\in U_1$ and pick two distinct minimizing $M_1$-geodesics $\ga_1,\ga_2$
from $x$ to points $q_1,q_2\in F$, respectively.
Let $\alpha\in(0,\pi]$ be the angle between $\ga_1$ and $\ga_2$ at~$x$.
For each $i\in\{1,2\}$, fix a point $p_i\in\ga_i\cap F\setminus\{q_i\}$.
Let $y$ be a point on $\ga_2$ close to~$x$.
By Corollary \ref{cor:1stvar} applied to $q_1,p_1,x,y$
we have
$$
 d_1(y,p_1)=d_1(x,p_1) - d_1(x,y)\cos\alpha + O(d_1(x,y)^2), \qquad y\in\ga_2,\ y\to x.
$$
Subtracting the identity $d_1(y,p_2)=d_1(x,p_2) - d_1(x,y)$ we obtain
$$
 \D_F^1(y)(p_1,p_2) = \D_F^1(x)(p_1,p_2) +(1-\cos\alpha) d_1(x,y) + O(d_1(x,y)^2), \quad y\in\ga_2,\ y\to x.
$$
Hence there exists a limit
\be\label{e:1-cos}
 \lim_{y\in\ga_2,y\to x} \frac{\D_F^1(y)(p_1,p_2)-\D_F^1(x)(p_1,p_2)}{d_1(x,y)} = 1 -\cos\alpha .
\ee

Let $\tilde x=\phi(x)$. Let $\tilde\ga_1$ and $\tilde\ga_2$ be minimizing $M_2$ geodesics
corresponding to $\ga_1$ and $\ga_2$ as in Lemma \ref{l:dirphi}.
In particular we have $q_i,p_i\in\tilde\ga_i$ for $i=1,2$.
Let $\tilde\alpha$ be the angle between $\tilde\ga_1$ and $\tilde\ga_2$ at $\tilde x$.
Similarly to \eqref{e:1-cos}, for points $\tilde y\in\tilde\ga_2$ close to $\tilde x$ we have
\be\label{e:1-costilde}
 \lim_{\tilde y\in\tilde\ga_2,\tilde y\to\tilde x} 
 \frac{\D_F^2(\tilde y)(p_1,p_2)-\D_F^2(\tilde x)(p_1,p_2)}{d_2(\tilde x,\tilde y)} = 1 -\cos\tilde\alpha .
\ee
If $y\in\ga_2\cap U_1$ and $\tilde y=\phi(y)$, then
$\tilde y\in\tilde\ga_2\cap U_2$ by Lemma \ref{l:dirphi}, and
$$
\D_F^1(y)(p_1,p_2)-\D_F^1(x)(p_1,p_2) = \D_F^2(\tilde y)(p_1,p_2)-\D_F^2(\tilde x)(p_1,p_2)
$$
by \eqref{e:phi}. This, \eqref{e:1-cos} and \eqref{e:1-costilde} imply that
\be\label{e:cosfrac}
\lim_{y\in\ga_2,y\to x} \frac{d_2(\tilde x,\phi(y))}{d_1(x,y)} = \frac{1 -\cos\tilde\alpha}{1 -\cos\alpha}
\ee
By switching the roles of $\ga_1$ and $\ga_2$, the same relation holds for $y\in\ga_1$, $y\to x$.
Hence
\be\label{e:limeq}
 \lim_{y\in\ga_1,y\to x} \frac{d_2(\tilde x,\phi(y))}{d_1(x,y)} 
 = \lim_{y\in\ga_2,y\to x} \frac{d_2(\tilde x,\phi(y))}{d_1(x,y)} .
\ee
Since $\ga_1$ and $\ga_2$ are arbitrary minimizing geodesics from $x$ to points of~$F$,
the limits in \eqref{e:limeq} do not depend on $\ga_1$ and $\ga_2$.
Thus there exists $\la=\la(x)>0$ such that,
for any minimizing $M_1$-geodesic $\ga$ from $x$ to a point of $F$,
\be\label{e:conformal}
 \lim_{y\in\ga,y\to x} \frac{d_2(\tilde x,\phi(y))}{d_1(x,y)} = \la .
\ee

We now show that $\la=1$.
%Let $\dir_1(x,F)$ and $\dir_2(\tilde x,F)$ denote the sets of
%initial directions of minimizing geodesics
%connecting $x$ and $\tilde x$ to points of~$F$ in $M_1$ and $M_2$, respectively,
%see Notation \ref{n:dir}.
%By Lemma \ref{l:dirvol}, the sets $\dir_1(x,F)$ and $\dir_2(\tilde x,F)$
%have positive measure in the respective unit spheres $S_xM_1$ and $S_{\tilde x}M_2$.
Lemma \ref{l:dirphi} allows us to a define map
$$
 \psi\co \dir_{M_1}(x,F) \to \dir_{M_2}(\tilde x,F)
$$
(see  Notation \ref{n:dir})
as follows: if $v\in\dir_1(x,F)$ is the initial direction
of a minimizing $M_1$-geodesic $\ga$ connecting $x$ to $q\in F$,
then $\psi(v)$ is the initial direction of
the $M_2$-geodesic $\tilde\ga$ corresponding to $\ga$ as in Lemma \ref{l:dirphi}.
By \eqref{e:cosfrac} and \eqref{e:conformal} we have
\be\label{e:anglefrac}
1 -\cos\angle_{g_2}(\psi(v_1),\psi(v_2)) = \la (1 -\cos\angle_{g_1}(v_1,v_2))
\ee
for all $v_1,v_2\in\dir_{M_1}(x,F)$,
where $\angle_{g_i}$ denotes the angle with respect to $g_i$.
Observe that
$$
 |v_1-v_2|^2 = 2(1 -\cos\angle(v_1,v_2))
$$
for any two unit vectors $v_1,v_2$ in a Euclidean space.
Hence \eqref{e:anglefrac} can be rewritten as
$$
 \|\psi(v_1)-\psi(v_2)\|^2_{g_2} = \la \|v_1-v_2\|^2_{g_1}, \qquad \forall v_1,v_2\in \dir_1(x,F).
$$
Thus $\psi$ it a homothetic map that multiplies all distances by a constant factor~$\sqrt\la$.
Hence $\psi$ can be extended to an affine map $\Psi\co T_xM_1\to T_{\tilde x}M_2$
with the same property:
$$
 \|\Psi(v_1)-\Psi(v_2)\|_{g_2} = \sqrt\la \|v_1-v_2\|_{g_1},\qquad \forall v_1,v_2\in T_xM_1.
$$
The set $\dir_{M_2}(\tilde x,F)$ is contained in the intersection
of the unit sphere $S_{\tilde x}M_2$ and the image $\Psi(S_xM_1)$.
The latter is a Euclidean sphere of radius $\sqrt\la$.
Suppose that $\la\ne 1$, then the intersection
of the two spheres is a submanifold of dimension at most $n-2$.
%and hence it has zero measure in $S_{\tilde x}M_2$.
On the other hand, by Lemma \ref{l:dirvol} the set $\dir_{M_2}(\tilde x,F)$
have a positive $(n-1)$-dimensional measure, a contradiction.
Thus $\la=1$.

Since $\la=1$, \eqref{e:anglefrac} implies that $\psi$ preserves angles:
\be\label{e:psi-preserves-angles}
 \angle_{g_2}(\psi(v_1),\psi(v_2)) = \angle_{g_1}(v_1,v_2)
\ee
for all $v_1,v_2\in\dir_{M_1}(x,F)$.

Recall that $\phi$ is a locally bi-Lipschitz map.
Hence, by Rademacher's theorem, it is differentiable almost everywhere.
Assume that $\phi$ is differentiable at the point $x$ that we are considering.
Then the definition of $\psi$, \eqref{e:conformal} and the fact that $\la=1$
imply that $d_x\phi(v)=\psi(v)$ for all $v\in\dir_{M_1}(x,F)$.

By Lemma \ref{l:dirvol}, $\dir_{M_1}(x,F)$ is a set positive measure in $S_xM_1$.
Hence it contains $n$ linearly independent vectors $v_1,\dots,v_n$.
For each $i=1,\dots,n$, we have 
$$
d_x\phi(v_i)=\psi(v_i) \in S_{\tilde x}M_2 .
$$
This and \eqref{e:psi-preserves-angles} imply that
$$
 g_2( d_x\phi(v_i),d_x\phi(v_j) ) = g_1(v_i,v_j) .
$$
%where $\langle\cdot,\cdot\rangle$ denotes the inner products
%on $T_{\tilde x}M$ and $T_xM_1$ determined by $g_2$ and~$g_1$.
Therefore $d_x\phi$ is a linear isometry between $T_xM_1$ and $T_{\tilde x}M_2$
equipped with their Euclidean structures determined by $g_1$ and $g_2$.

Thus the derivative $d_x\phi$ is a linear isometry for almost all $x\in U_1$.
Hence $\phi$ is a locally 1-Lipschitz map with respect to $d_1$ and $d_2$.
The same argument applies to $\phi^{-1}$, 
therefore $\phi$ is a locally distance preserving map.
%, that is,
%every point $x\in U_1$ has a neighborhood $U\subset U_1$ such that
%$d_1(y,z)=d_2(\phi(y),\phi(z))$ for all $y,z\in U$.
By the Myers-Steenrod theorem (\cite{MS}, see also \cite[Ch.~5, Theorem 18]{Pe})
this implies that $\phi$ is a Riemannian isometry.
This finishes the proof of Theorem \ref{t:reconstruction}.

\subsection{Stability}
Here we show that the determination of a manifold by distance difference data
is stable in a suitable sense.
We add a number of simplifying assumptions to the setting of Theorem \ref{t:reconstruction}.
First, the unknown manifold $M$ is assumed compact, and we assume a priory bounds
on its diameter, sectional curvature, and injectivity radius.
That is, the manifold belongs to the class $\mathcal M(n,D,K,i_0)$
defined in the introduction.
%Note that this class is pre-compact in Gromov-Hausdorff topology,
%moreover on its closure the Gromov-Hausdorff topology coincides with
%the $C^{1,\alpha}$ topology, see \cite[Ch.~10]{Pe}.
Second, we a dealing with determination of the whole manifold $M$
(rather than a subset $U$) by the full distance difference data $\D_F(M)$.
Third, we assume that the metric of the observation domain $F$ is known.

To bring a ``practical'' flavour to our set-up,
imagine that we know the distance difference data $\D_F(M)$
with some error. More precisely, suppose that we are given
approximate distance difference data $\{\widetilde\D_F(x_i)\}$,
where $\{x_i\}$ is $\delta$-net in $M$ and $\widetilde\D_F(x_i)$
is the distance difference function $\D_F(x_i)$
measured with an absolute error~$\delta$.
This implies that the set $\{\widetilde\D_F(x_i)\}$ is a $3\de$-approximation
of the distance difference data $\D_F(M)$ in the sense of Hausdorff
distance between subsets of $\C(F\times F)$.
Our goal is to show that these approximate data determine
the manifold up to an error $\ep(\delta)$ in the Gromov--Hasdorff sense,
where $\ep(\de)\to 0$ as $\de\to 0$.
Explicit bounds on $\ep(\de)$ and an actual reconstruction procedure
is outside the scope of this paper.
The precise statement of our stability result is the following.

\begin{proposition}\label{p:stability}
For every $\ep>0$, $n\ge 2$ and $D,K,i_0,\rho_0>0$ there exists $\de>0$ such that the following holds.
Let $M_1=(M_1^n,g_1)$ and $M_2=(M_2^n,g_2)$ belong to $\mathcal M(n,D,K,i_0)$.
Assume that $M_1$ and $M_2$  share a nonempty subset $F$ which 
is open in both manifolds,
they induce the same topology and the same differential structure on $F$,
and $g_1|_F=g_2|_F$.
In addition, assume that $F$ contains a geodesic ball of radius $\rho_0$.

For $i=1,2$, let $\D_F^i$ denote the distance difference representation
of $M_i$ % in the space $\C(F\times F)$
and suppose that
$$%\be\label{e:d_H}
 d_H(\D_F^1(M_1) , \D_F^2(M_2)) < \de
$$%\ee
where $d_H$ is the Hausdorff distance in $\C(F\times F)$.
Then $d_{GH}(M_1,M_2)<\ep$ where $d_{GH}$ denotes the Gromov-Hausdorff distance.
\end{proposition}

\begin{proof}[Proof sketch]
It suffices to consider the case when $F$ is a geodesic ball of radius $\rho_0$.
We may further assume that $\rho_0<\min\{i_0/2,K^{-1/2}\}$.
Then every ball of radius $\rho_0$ is convex.
In particular, $M_1$ and $M_2$ induce the same distances in~$F$.

Suppose that the statement of the proposition is false.
Then there exist a sequence $\de_k\to 0$ and sequences
$\{M_{1,k}\}$, $\{M_{2,k}\}$, $k=1,2,\dots$,
of Riemannian manifolds from $\mathcal M(n,D,K,i_0)$ and geodesic balls
$F_k=B_{\rho_0}(q_k)\subset M_{1,k}\cap M_{2,k}$ satisfying
the assumptions of the proposition for $\de=\de_k$
and such that the Gromov-Hausdorff distance $d_{GH}(M_{1,k},M_{2,k})$
is bounded away from~0.

Recall that the class $\mathcal M(n,D,K,i_0)$ is pre-compact in Gromov-Hausdorff topology,
see \cite[Ch.~10]{Pe}.
Hence, by passing to a subsequence if necessary, we may assume that
metric pairs $(M_{1,k},F_k)$ and $(M_{2,k},F_k)$ converge to 
some pairs $(\bar M_1,\bar F)$ and $(\bar M_2,\bar F)$
in the Gromov-Hausdorff sense,
where $\bar M_1$  and $\bar M_2$ are compact metric spaces
and $\bar F$ is a ball of radius $\rho_0$ in both $\bar M_1$ and $\bar M_2$.
Furthermore $\bar M_1$ and $\bar M_2$ induce the same metric on~$\bar F$.

Let $\D_{\bar F}^1$ and $\D_{\bar F}^2$ be distance difference representations
of $\bar M_1$ and $\bar M_2$ in $\C(\bar F\times\bar F)$.
The distance functions of points of $\bar M_1$ and $\bar M_2$ are in a natural sense
limits of distance functions in $M_{1,k}$ and $M_{2,k}$.
Since $d_H(\D_{F_k}^1(M_{1,k}) , \D_{F_k}^2(M_{2,k})) < \de_k\to 0$,
it follows that $\D_{\bar F}^1(M_1)=\D_{\bar F}^2(M_2)$.

Since the convergent spaces $M_{1,k}$ and $M_{2,k}$ are from $\mathcal M(n,D,K,i_0)$,
the limit spaces $\bar M_1$  and $\bar M_2$ are $n$-dimensional Alexandrov spaces
with two-sided curvature bounds.
Such spaces are $C^{1,\alpha}$ Riemannian manifolds for every $\alpha\in(0,1)$, see \cite{ABN},
moreover the Gromov-Hausdorff convergence is actually a $C^{1,\alpha}$ convergence,
see \cite[Ch.~10]{Pe}. One can see that all ingredients of the proof
of Theorem \ref{t:reconstruction} work for such spaces
$\bar M_1$  and $\bar M_2$.
Then Theorem \ref{t:reconstruction} implies that $\bar M_1$  and $\bar M_2$ are isometric,
contrary to the assumption that the Gromov-Hausdorff distance $d_{GH}(M_{1,k},M_{2,k})$
is bounded away from~0.
\end{proof}

\section{Distance differences on the boundary}
\label{sec:boundary}

In this section we consider a variant of the problem where $M=(M^n,g)$
is a complete, connected Riemannian manifold with boundary
and $F=\pd M$.
%The boundary is not assumed convex.
The distance $d(x,y)$ between $x,y\in M$ is the length of a shortest path
(possibly touching the boundary)
from $x$ to~$y$ in~$M$.
%(that may bend along the boundary on some intervals).
This distance defines the distance difference representation
$\D_F\co M\to \C(F\times F)$ in the usual way.

We attach a collar $\tilde F:=\pd M\times [0,+\infty)$ to $M$ along the boundary
and equip the resulting manifold $\tilde M=M\cup\tilde F$ with a complete Riemannian
metric $\tilde g$ extending $g$.
%Such an extension exists for any $M$, 
%however it depends on the $C^\infty$-jet of $g$ on $\pd M$.
Now $M$ is a subset of $\tilde M$ bounded by a smooth hypersurface $F=\pd M$.
Denote by $\tilde d$ the distance in $\tilde M$ induced by $\tilde g$
and
by $\D_{\tilde F}$ the distance difference representation
of $\tilde M$ in $\C(\tilde F\times\tilde F)$.
Note that $d(x,y)\ge \tilde d(x,y)$ for all $x,y\in M$.

\begin{lemma}\label{l:DtildeF}
In the above notation,
\be\label{e:DtildeF}
 \|\D_{\tilde F}(x)-\D_{\tilde F}(y)\| \le \|\D_F(x)-\D_F(y)\|
\ee
for all $x,y\in M$.
\end{lemma}

\begin{proof}
Let $x,y\in M$ and $p,q\in\tilde F$.
Consider a minimizing geodesic $[xp]$ of $\tilde M$
%Since $x\in M$ and $p$ is outside the interior of $M$,
%we have $[xp]\cap F\ne\emptyset$.
%Let 
and let $x_1$ be the first point where this geodesic crosses the boundary,
i.e., $x_1$ is the point of $[xp]\cap F$ nearest to~$x$.
The interval $[xx_1]$ of $[xp]$ is a minimizing geodesic
of $\tilde M$ contained in $M$, hence it is a shortest path in~$M$.
Thus $d(x,x_1)=\tilde d(x,x_1)$.
By the triangle inequality,
$$
 \tilde d(y,p) \le \tilde d(y,x_1)+\tilde d(x_1,p)
 = \tilde d(y,x_1)+\tilde d(x,p) - \tilde d(x,x_1) .
$$
Since $\tilde d(y,x_1)\le d(y,x_1)$ and $\tilde d(x,x_1)=d(x,x_1)$,
it follows that
$$
 \tilde d(y,p) \le \tilde d(x,p) + d(y,x_1)- d(x,x_1) .
$$
Similarly,
$$
 \tilde d(x,q) \le \tilde d(y,q) + d(x,y_1)- d(y,y_1)
$$
where $y_1$ is the point on $[yq]\cap F$ nearest to~$y$.
Summing the last two inequalities yields that
$$
\tilde d(y,p) + \tilde d(x,q) \le \tilde d(x,p) + \tilde d(y,q) + \D_F(y)(x_1,y_1)-\D_F(x)(x_1,y_1)
$$
or, equivalently,
$$
 \D_{\tilde F}(y)(p,q) - \D_{\tilde F}(x)(p,q) \le \D_F(y)(x_1,y_1)-\D_F(x)(x_1,y_1) .
$$
The right-hand side is bounded by $\|\D_F(y)-\D_F(x)\|$, therefore
$$
\D_{\tilde F}(y)(p,q) - \D_{\tilde F}(x)(p,q) \le \|\D_F(y)-\D_F(x)\| .
$$
Since $p$ and $q$ are arbitrary points of $\tilde F$, \eqref{e:DtildeF} follows.
\end{proof}

Now we prove an analogue of Theorem \ref{t:inverse} for boundary distances.

\begin{proposition}\label{p:dM-bilip}
Let $M^n$ be a complete, connected Riemannian manifold with boundary,
$n\ge 2$, and $F=\pd M$.
Then the map $\D_F\co M\to\C(F\times F)$ is a homeomorphism onto its image
and the inverse map $\D_F^{-1}\co\D_F(M)\to M$ is locally Lipschitz.
\end{proposition}

\begin{proof}
Consider $\tilde M=(\tilde M,\tilde g)$ and $\tilde F$ constructed as above.
By Theorem \ref{t:inverse}, the map $\D_{\tilde F}\co \tilde M \to \C(\tilde F\times\tilde F)$
is injective and its inverse $\D_{\tilde F}^{-1}\co \D_{\tilde F}(\tilde M)\to \tilde M$
is locally Lipschitz.
Consider the map $I\co\D_{\tilde F}(M)\to\D_F(M)$ defined by
$I(\D_{\tilde F}(x))=\D_{F}(x)$, $x\in M$.
By Lemma \ref{l:DtildeF}, $I$ does not decrease distances.
%
%
%That is, for every $x_0\in M$ there exist $\tilde\de>0$ and $\tilde C>0$ such that
%$$
% \tilde d(x,y) \le \tilde C \|\D_{\tilde F}(x) -\D_{\tilde F}(y)\| %\le \tilde C \|\D_{F}(x) -\D_{F}(y)\|
%$$
%for all $x,y\in M$ such that $\|\D_{\tilde F}(x) -\D_{\tilde F}(x_0)\|<\tilde\de$
%and $\|\D_{\tilde F}(y) -\D_{\tilde F}(x_0)\|<\tilde\de$.
%This and Lemma \ref{l:DtildeF} imply that
%$$
% \tilde d(x,y) \le \tilde C \|\D_{F}(x) -\D_{F}(y)\| %\le \tilde C \|\D_{F}(x) -\D_{F}(y)\|
%$$
%for all $x,y\in M$ such that $\|\D_{F}(x) -\D_{F}(x_0)\|<\tilde\de$
%and $\|\D_{F}(y) -\D_{F}(x_0)\|<\tilde\de$.
%
Hence $\D_F=I\circ\D_{\tilde F}|_M$ is injective
and its inverse is locally Lipschitz with respect to
$\tilde d|_{M\times M}$.
Since $\tilde d|_{M\times M}$ is locally bi-Lipschitz equivalent to $d$,
the result follows.
\end{proof}

The next Proposition \ref{p:boundary} 
asserts that $(M,g)$ is uniquely determined by its distance difference data
on the boundary provided that the boundary is nowhere concave.
We say that the boundary $\pd M$ of $M$ is
\textit{nowhere concave} if for every $x\in\pd M$ the second fundamental form
of $\pd M$ at $x$ with respect to the inward-pointing normal vector 
has at least one positive eigenvalue.

The nowhere concavity ensures that the boundary distance function
uniquely determines the $C^\infty$-jet of the metric at the boundary
(see the proof below).
This is the only implication of nowhere concavity that we need;
any other condition allowing the determination of the $C^\infty$-jet
would work just as well.
It is plausible that the boundary assumption in Proposition \ref{p:boundary} can be removed but it is essential for our method to work.

\begin{remark}
Recently M. V. de Hoop and T. Saksala \cite{HS}, using a different approach,
obtained a similar result under weaker {\it visibility} assumptions
on the boundary.
For compact manifolds Proposition \ref{p:boundary} can be deduced
from the results of~\cite{HS}.
\end{remark}

\begin{proposition}\label{p:boundary}
Let $M_1=(M_1^n,g_1)$ and $M_2=(M_2^n,g_2)$, $n\ge 2$, be complete, connected Riemannian manifolds
with a common boundary $F=\pd M_1=\pd M_2$,
which is nowhere concave in both $M_1$ and $M_2$.

Let $\D_F^i$ denote the distance difference representation
of $M_i$ in $\C(F\times F)$, $i=1,2$, and
assume that $\D_F^1(M_1)=\D_F^2(M_2)$.
Then $M_1$ and $M_2$ are isometric via an isometry
fixing the boundary.
\end{proposition}

\begin{proof}
Define $\phi \co M_1\to M_2$ by $\phi=(\D_F^2)^{-1}\circ \D_F^1$.
We are going to show that $\phi$ is an isometry fixing the boundary.

Proposition \ref{p:dM-bilip} implies that $\phi$ is well defined and
it is a homeomorphism between $M_1$ to $M_2$.
Hence it maps boundary to boundary, i.e., $\phi(F)=F$.
Moreover, $\phi|_F=\mbox{id}_F$.
Indeed, a point $x\in F$ is determined by $\D^1_F(x)$ as
the unique minimum point of a function $\D^1_F(x)(\cdot,x_0)$ where
$x_0$ is an arbitrary fixed point of~$F$.
If $x'=\phi(x)$ then $\D^1_F(x)=\D^2_F(x')$ and hence
$x'$ is the unique minimum point of the same function.
Thus $\phi(x)=x$ and $\D_F^1(x)=\D_F^2(x)$ for all $x\in F$.

Let $d_1$ and $d_2$ denote the distances in $M_1$ and $M_2$, resp.
Then $d_1|_{F\times F}=d_2|_{F\times F}$.
Indeed, for any $x,y\in F$,
$$
 d_1(x,y)= \D^1_F(x)(y,x)= \D^2_F(x)(y,x) = d_2(x,y)
$$
where the second identity follows from the fact that $\phi(x)=x$.
Since the boundary $F$ is nowhere concave in $M_i$, $i=1,2$,
the boundary distance function $d_i|_{F\times F}$ determines
the $C^\infty$-jet of $g_i$ at the boundary (in e.g.\ boundary normal coordinates), see \cite{LSU,UW,Z}.
Thus $g_1$ and $g_2$ have the same $C^\infty$-jet at the boundary.

As in the beginning of this section,
we attach a collar $\tilde F:=F\times[0,+\infty)$ to the boundary of $M_1$
and equip $\tilde F$ with a Riemannian metric $g$ smoothly extending~$g_1$.
Since $g_1$ and $g_2$ have the same $C^\infty$-jet at the boundary,
the metric $g$ smoothly extends $g_2$ as well.
Thus we obtain two complete Riemannian manifolds
$(\tilde M_i,\tilde g_i)$, $i=1,2$, where $\tilde M_i=M_i\cup\tilde F$
and $\tilde g_i=g_i\cup g$.

Let $d_{\tilde F}$ denote the length distance in $\tilde F$ (induced by $g$),
$\tilde d_i$ the distance in $\tilde M_i$,
and $D_{\tilde F}^i$ the distance difference representations of $\tilde M_i$
in $\C(\tilde F\times\tilde F)$, $i=1,2$.
Note that the distances $d_i$ and $d$ may differ from the respective restrictions of $\tilde d_i$
since they are determined by shortest paths in possibly non-convex subsets
$M_i$ and $\tilde F$.
We are going to show that $D_{\tilde F}^1(M_1)= D_{\tilde F}^2(M_2)$.
More precisely, we prove that
\be\label{e:tildeF}
 \D^1_{\tilde F}(x) = \D^2_{\tilde F}(\phi(x))
\ee
for all $x\in M_1$.

Let $x\in M_1$ and $y\in\tilde F$.
The distance $\tilde d_1(x,y)$ equals the infimum of lengths of
piecewise smooth paths in $\tilde M_1$ crossing $F$ finitely many times.
Hence
\be\label{e:tilded}
 \tilde d_1(x,y) = \inf %_{\{x_j\},\{y_j\}\subset F}
 \left\{ d_1(x,y_0) + \sum_{j=1}^N \bigl(d_{\tilde F}(y_{j-1},x_j)+d_1(x_j,y_j)\bigr) + d_{\tilde F}(y_N,y) \right\}
%: x_i,y_i\in F \right\}
\ee
where the infimum is taken over all finite sequences $y_0,\dots,y_N$ and $x_1,\dots,x_N$ in~$F$.
For $x'=\phi(x)\in M_2$, the distance $\tilde d_2(x',y)$ is given by the same formula
with $x'$ in place of $x$ and $d_2$ in place of $d_1$.
By the definition of $\phi$ there is a constant $\tau=\tau(x)\in\R$ such that
$d_2(x',y_0)=d_1(x,y_0)+\tau$ for all $y_0\in F$.
Since $d_1|_{F\times F}=d_2|_{F\times F}$, we have 
$d_2(x_j,y_j)=d_1(x_j,y_j)$ for any $x_j,y_j\in F$.
This and \eqref{e:tilded} imply that $\tilde d_2(x',y)=\tilde d_1(x,y)+\tau$.
Since $\tau$ does not depend on $y$, this implies that
$$
 \tilde d_2(x',y) - \tilde d_2(x',z) = \tilde d_1(x,y) - \tilde d_1(x,z)
$$
for all $y,z\in\tilde F$,
and \eqref{e:tildeF} follows.
By Theorem \ref{t:reconstruction}, \eqref{e:tildeF} implies that $\phi$ is an isometry
from $M_1$ to $M_2$.
This finishes the proof of Proposition \ref{p:boundary}.
\end{proof}

\end{document}